\documentclass[psreqno,reqno,11pt]{amsart}
\usepackage{eurosym}
\usepackage{amssymb}
\usepackage{tikz}
\usepackage{amsmath}
\usepackage{tikz}
\usepackage{pgflibraryplotmarks}
\usepackage{wasysym}
\usepackage{stmaryrd}
\usepackage[english]{babel}

\usepackage{hyperref}

\theoremstyle{plain}
\newtheorem{theorem}{Theorem}[section]
\newtheorem{proposition}[theorem]{Proposition}
\newtheorem{lemma}[theorem]{Lemma}
\newtheorem{corollary}[theorem]{Corollary}

\theoremstyle{definition}
\newtheorem{definition}[theorem]{Definition}
\newtheorem{notation}[theorem]{Notation}
\newtheorem{example}[theorem]{Example}
\newtheorem{remark}[theorem]{Remark}

\allowdisplaybreaks

\begin{document}
	\def\N{\mathbb{N}}
	\def\Z{\mathbb{Z}}
	\def\R{\mathbb{R}}
	\def\C{\mathbb{C}}
	\def\M{\mathcal{M}}
	
	\title[An Archimedean Vector Lattice Functional Calculus]{
		An Archimedean Vector Lattice Functional Calculus For Semicontinuous Positively Homogeneous Functions}
	\author{C. Schwanke}
	\address{Department of Mathematics and Applied Mathematics, University of Pretoria, Private Bag X20, Hatfield 0028, South Africa}
	\email{cmschwanke26@gmail.com}
	\date{\today}
	
	\begin{abstract}
		We develop a functional calculus on Archimedean vector lattices for semicontinuous positively homogeneous real-valued functions defined on $\R^n$ which are bounded on the unit sphere. It is further shown that this semicontinuous Archimedean vector lattice functional calculus extends the existing continuous Archimedean vector lattice functional calculus by Buskes, de Pagter, and van Rooij. We further utilize saddle representations of continuous positively homogeneous functions to provide concrete formulas, for functions abstractly defined via the continuous functional calculus, which are completely in terms of vector lattice operations. Finally, we provide some examples to illustrate the utility of the theory presented.
	\end{abstract}
	
	\maketitle

	\section{Introduction}\label{S:intro}
	
	The idea of naturally defining expressions of the form $h(f_1,\dots,f_n)$, where $h$ is a real-valued function defined on $\R^n$, and $f_1,\dots,f_n$ are elements of a vector lattice, is called functional calculus. The study of functional calculus on Archimedean vector lattices dates back to the 1991 paper \cite{BusdPvR} by Buskes, de Pagter, and van Rooij. Given a uniformly complete Archimedean vector lattice $E$, these authors defined $h(f_1,\dots,f_n)$ for any continuous positively homogeneous function $h\colon\R^n\to\R$ and any $f_1,\dots,f_n\in E$. Functional calculus allows for useful expressions in functional analysis, such as the root mean square $\frac{1}{\sqrt{n}}\sqrt{\sum_{k=1}^nf_k^2}$ and the geometric mean $\sqrt[n]{\prod_{k=1}^{n}|f_k|}$, to be defined in uniformly complete Archimedean vector lattices. This achievement  has considerably enhanced the development of vector lattice theory in numerous directions. Indeed, functional calculus with respect to the scalar multiple of the square mean $\sqrt{x^2+y^2}$ in two variables was utilized in \cite{BusSch2} to construct a tensor product in the category of Archimedean complex vector lattices. Functional calculus also led to the extension of the classical Cauchy-Schwarz, H\"older, and Minkowski inequalities in \cite{BusSch3}, and was used to characterize bounded orthogonally additive polynomials in \cite{BusSch4, Kusa, Sch}.
	
	While the introduction of functional calculus in Archimedean vector lattices has proven to be tremendously beneficial, it is not a theory without its challenges. One limitation of the theory is its level of abstractness; that is, if $h\colon\R^n\to\R$ is continuous and positively homogeneous, $E$ is a uniformly complete Archimedean vector lattice, and $f_1,\dots,f_n\in E$, then while we know that $h(f_1,\dots,f_n)$ defined by the functional calculus of \cite{BusdPvR} exists, it is typically not known what $h(f_1,\dots,f_n)$ ``looks like" in terms of vector lattice operations. Progress towards providing explicit formulas for these expressions were made in \cite{BusSch, Kus} for convex and concave continuous positively homogeneous functions, but extending such formulas to more general continuous positively homogeneous functions has seemed more difficult.
	
	However, certain representations of positively homogeneous functions presented in the literature on nonsmooth analysis and optimization theory represent continuous positively homogeneous functions purely in terms of vector lattice operations (see e.g. \cite{DR, GorTraf}). These results provide a natural avenue for obtaining explicit formulas for \textit{every} real-valued continuous positively homogeneous function defined on $\R^n$. We refer to these representations as \textit{inf-sublinear representations} and \textit{sup-superlinear} representations, see Definition~\ref{D: inf-sublinear and sup-superlinear}.
	
	Interestingly, inf-sublinear and sup-superlinear representations exist not only for all continuous real-valued positively homogeneous functions on $\R^n$ but all \textit{semicontinuous} real-valued positively homogeneous functions on $\R^n$ which are bounded on the unit sphere. This observation enables us to provide a theory of functional calculus for semicontinuous real-valued positively homogeneous functions on $\R^n$ which are bounded on the unit sphere in this paper. Our functional calculus for semicontinuous positively homogeneous functions extends the original Archimedean vector lattice functional calculus in \cite{BusdPvR} and provides explicit formulas solely in terms of vector lattice operations.

	We build our semicontinuous Archimedean vector lattice functional calculus in four steps. First, we construct a Dedekind complete vector lattice functional calculus for sublinear maps and superlinear maps in Section~\ref{S: DC phi fun cal}. In Section~\ref{S: Arch phi fun cal} we then, when possible, extend this functional calculus to a broader range of Archimedean vector lattices. Our results from Section~\ref{S: DC phi fun cal} enable us to exploit the aforementioned inf-sublinear and sup-superlinear representations to develop in Section~\ref{S: h fun cal} a functional calculus for real-valued semicontinuous positively homogeneous functions on $\R^n$, which are bounded on the unit sphere, in any Dedekind complete vector lattice. In Section~\ref{S: non DC h fun cal}, we extend, when feasible, the functional calculus in Section~\ref{S: h fun cal} to certain Archimedean vector lattices which may not be Dedekind complete. We then prove in Section~\ref{S: comparisonto'91} that our semicontinuous functional calculus agrees with the original Archimedean vector lattice functional calculus for continuous positively homogeneous functions introduced in \cite{BusdPvR}, and we provide alternate formulas for this continuous version of functional calculus on Dedekind complete vector lattices in terms of saddle representations (as defined in  \cite[Section~2]{GorTraf}) in Section~\ref{S: saddlerepresnentations}. Finally, in Section~\ref{S: examples} we provide some specific examples which further illustrate how the definitions and results of this paper work.
	
	We proceed with some preliminaries.
	
	\section{Preliminaries}\label{S: Prelims}
	
	The reader is referred to the standard texts (e.g. \cite{AB,LuxZan1,Zan2}) for any unexplained terminology or basic theory regarding vector lattices. Throughout, $\mathbb{R}$ is used for the real numbers, and the set of strictly positive integers is denoted by $\mathbb{N}$.
	
	We remind the reader of the functions which are central to the research in this paper.
	
	\begin{definition}
		Given $n\in\N$, a function $h\colon\R^n\to\R$ is called \textit{positively homogeneous} if for all $\lambda\in[0,\infty)$ and every $x_1,\dots,x_n\in\R$,
		\[
		h(\lambda x_1,\dots,\lambda x_n)=\lambda h(x_1,\dots,x_n).
		\]
	\end{definition}
	
	Next we provide some notation that was used in \cite{BusdPvR} to develop the original Archimedean vector lattice functional calculus.
	
	\begin{notation}
		Given $n\in\N$, an Archimedean vector lattice $E$, and $f_1,\dots,f_n\in E$, we denote by $\langle f_1,\dots,f_n\rangle$ the vector sublattice of $E$ generated by $\{f_1,\dots,f_n\}$, and the set of all real-valued vector lattice homomorphisms defined on $E$ is written as $H(E)$.
	\end{notation}
	
	We next present the Archimedean vector lattice functional calculus developed by Buskes, de Pagter, and van Rooij in \cite{BusdPvR}.
	
	\begin{definition}\cite[Definition~3.1]{BusdPvR}
		Fix $n\in\N$. Consider an Archimedean vector lattice $E$, and let $f_1,\dots,f_n\in E$. Suppose $h\colon\R^n\to\R$ is a continuous positively homogeneous function. If there exists $g\in E$ such that for every $\omega\in H(\langle f_1,\dots,f_n,g\rangle)$,
		\[
		h\bigl(\omega(f_1),\dots,\omega(f_n)\bigr)=\omega(g),
		\]
		then we define $h(f_1,\dots,f_n):=g$. We refer to this definition as the \textit{continuous Archimedean vector lattice functional calculus} in this paper.
	\end{definition}
	
	As we develop an Archimedean vector lattice functional calculus for semicontinuous positively homogeneous functions bounded on the unit sphere, we will frequently make use of the following convenient notation.
	
	\begin{notation}
		Given $n\in\N$, $\bar{x}\in\R^n$, and $i\in\{1,\dots,n\}$, we as usual denote the $i$th coordinate of $\bar{x}$ by $x_i$ throughout this paper. Thus we have $\bar{x}=(x_1,\dots,x_n)$ for all $\bar{x}\in\R^n$. As usual, we also define
		\[
		\|\bar{x}\|:=\sqrt{\sum_{i=1}^{n}x_i^2}\qquad (\bar{x}\in\R^n).
		\]
		Moreover, for $k\in\{1,\dots,n\}$, we denote the element in $\R^n$ with $1$ in the $k$th coordinate and $0$ in all other coordinates by $\overline{e^{(k)}}$. We thus have that $\overline{e^{(k)}}=(e_1^{(k)},\dots,e_n^{(k)})$, where $e^{(k)}_k=1$, and for all $i\in\{1,\dots,n\}\setminus\{k\}$, $e^{(k)}_i=0$. We also denote the zero element of $\R^n$ by $\bar{0}$ throughout.
	\end{notation}
	
	\section{A Dedekind Complete Vector Lattice Functional Calculus for Sublinear and Superlinear Maps}\label{S: DC phi fun cal}
	
	Fix $n\in\N$, and let $\phi\colon\R^n\to\R$ be a sublinear map. By the Minkowski duality isomorphism (see e.g. \cite{H-U&L}), $\phi$ uniquely corresponds to its subdifferential $\underline{\partial}\phi(\bar{0})$ at $\bar{0}$, defined as
	\[
	\underline{\partial}\phi(\bar{0}):=\left\{\bar{a}\in\R^n\ :\ \text{for all}\ \bar{x}\in\R^n,\ \sum_{i=1}^{n}a_ix_i\leq \phi(\bar{x})\right\},
	\]
	which is a nonempty, convex, and compact set. Here we have
	\begin{equation}\label{eq: phi max}
		\phi(\bar{x})=\underset{\bar{a}\in\underline{\partial}\phi(\bar{0})}{\max}\ \sum_{i=1}^{n}a_ix_i\qquad (\bar{x}\in\R^n).
	\end{equation}
	
	Likewise, if $\psi\colon\R^n\to\R$ is a superlinear map, then $\psi$ is uniquely connected to its superdifferential $\overline{\partial}\psi(\bar{0})$ at $\bar{0}$, defined by
	\[
	\overline{\partial}\psi(\bar{0}):=\left\{\bar{a}\in\R^n\ :\ \text{for all}\ \bar{x}\in\R^n,\ \sum_{i=1}^{n}a_ix_i\geq \psi(\bar{x})\right\},
	\]
	which also is a nonempty, convex, and compact set. Moreover, we have
	\begin{equation}\label{eq: psi min}
		\psi(\bar{x})=\underset{\bar{a}\in\overline{\partial}\psi(\bar{0})}{\min}\ \sum_{i=1}^{n}a_ix_i\qquad (\bar{x}\in\R^n).
	\end{equation}
	
\begin{remark}
Let $n\in\N$, and consider a sublinear map $\phi\colon\R^n\to\R$. One can show, using \eqref{eq: phi max} and the compactness of $\underline{\partial}\phi(\bar{0})$, that there exists $M>0$ such that for all $\bar{x}\in\R^n$ with $\|\bar{x}\|=1$, $|\phi(\bar{x})|\leq M$. Exploiting this result and using hyperspherical coordinates, it is then readily proved that $\phi$ is continuous at $\bar{0}$. From here, one can use the sublinearity of $\phi$ to prove that $\phi$ is continuous on $\R^n$. Consequently, if $\psi\colon\R^n\to\R$ is a superlinear map, then $\psi$ is continuous on $\R^n$ as well.
\end{remark}
	
	Next consider a Dedekind complete vector lattice $E$, fix $n\in\N$, put $f_1,\dots,f_n\in E$, let $\phi\colon\R^n\to\R$ be a sublinear map, and let $\psi\colon\R^n\to\R$ be a superlinear map. We prove that the expressions
	\[
	\underset{\bar{a}\in\underline{\partial}\phi(\bar{0})}{\sup}\ \sum_{i=1}^{n}a_if_i
	\]
	and
	\[
	\underset{\bar{a}\in\overline{\partial}\psi(\bar{0})}{\inf}\ \sum_{i=1}^{n}a_if_i,
	\]
	related to \eqref{eq: phi max} and \eqref{eq: psi min}, always exist in $E$.
	
	\begin{proposition}\label{P: dc implies phic}
		Fix $n\in\N$. Let $E$ be a Dedekind complete vector lattice, and put $f_1,\dots,f_n\in E$.
		\begin{itemize}
			\item[(i)] If $\phi\colon\R^n\to\R$ is a sublinear map, then $\underset{\bar{a}\in\underline{\partial}\phi(\bar{0})}{\sup}\ \sum_{k=1}^{n}a_kf_k$ exists in $E$.
			\item[(ii)] If $\psi\colon\R^n\to\R$ is a superlinear map, then $\underset{\bar{a}\in\overline{\partial}\psi(\bar{0})}{\inf}\ \sum_{k=1}^{n}a_kf_k$ exists in $E$.
		\end{itemize}
	\end{proposition}
	
	\begin{proof}
		We only prove part (i) since both parts are proven in a similar fashion. To this end, let $\phi\colon\R^n\to\R$ be a sublinear map, let $\bar{a}\in\underline{\partial}\phi(\bar{0})$, and let $k\in\{1,\dots,n\}$.
		
		If $a_k\geq 0$, we have
		\begin{align*}
			|a_k|&=a_1\cdot 0+\cdots+a_{k-1}\cdot 0+a_k\cdot 1+a_{k+1}\cdot 0+\cdots+a_n\cdot 0\\
			&\leq\underset{\bar{a}\in\underline{\partial}\phi(\bar{0})}{\sup}\ \sum_{i=1}^{n}a_ie^{(k)}_i\\
			&=\phi\left(\overline{e^{(k)}}\right)\\
			&\leq\phi\left(\overline{e^{(k)}}\right)\vee\phi\left(-\overline{e^{(k)}}\right).
		\end{align*}
		
		On the other hand, if $a_k<0$, then
		\begin{align*}
			|a_k|&=a_1\cdot 0+\cdots+a_{k-1}\cdot 0+a_k(-1)+a_{k+1}\cdot 0+\cdots+a_n\cdot 0\\
			&\leq\underset{\bar{a}\in\underline{\partial}\phi(\bar{0})}{\sup}\ \sum_{i=1}^{n}a_i(-e^{(k)}_i)\\
			&=\phi\left(-\overline{e^{(k)}}\right)\\
		&\leq\phi\left(\overline{e^{(k)}}\right)\vee\phi\left(-\overline{e^{(k)}}\right).
		\end{align*}
		Hence we have
		\begin{equation}\label{eq: a<phi v phi}
			|a_k|\leq\phi\left(\overline{e^{(k)}}\right)\vee\phi\left(-\overline{e^{(k)}}\right).
		\end{equation}
		Utilizing \eqref{eq: a<phi v phi}, we obtain
		\begin{align*}
			\sum_{i=1}^{n}a_if_i&\leq\sum_{i=1}^{n}|a_i||f_i|\\
			&\leq\left(\phi\left(\overline{e^{(k)}}\right)\vee\phi\left(-\overline{e^{(k)}}\right)\right)\sum_{i=1}^{n}|f_i|.
		\end{align*}
		Thus the nonempty set $\left\{\sum_{i=1}^{n}a_if_i\ :\ \bar{a}\in\underline{\partial}\phi(\bar{0})\right\}$ is bounded above. Since $E$ is Dedekind complete, we conclude that $\underset{\bar{a}\in\underline{\partial}\phi(\bar{0})}{\sup}\ \sum_{i=1}^{n}a_if_i$ exists in $E$.
	\end{proof}
	
	With Proposition~\ref{P: dc implies phic} in place, we make use of the expressions \eqref{eq: phi max} and \eqref{eq: psi min} to naturally define a functional calculus for sublinear and superlinear maps on Dedekind complete vector lattices.
	
	\begin{definition}\label{D: DC sblin and splin fun cal}
		Let $E$ be a Dedekind complete vector lattice, fix $n\in\N$, and put $f_1,\dots,f_n\in E$.
		\begin{itemize}
			\item[(i)] Given a sublinear map $\phi\colon\R^n\to\R$, we define
			\[
				\phi(f_1,\dots,f_n):=\underset{\bar{a}\in\underline{\partial}\phi(\bar{0})}{\sup}\ \sum_{i=1}^{n}a_if_i.
			\]
			\item[(ii)] For a superlinear map $\psi\colon\R^n\to\R$, we write
			\[
				\psi(f_1,\dots,f_n):=\underset{\bar{a}\in\overline{\partial}\psi(\bar{0})}{\inf}\ \sum_{i=1}^{n}a_if_i.
			\]
		\end{itemize}
	\end{definition}
	
	Our next theorem illustrates a convenient interchange property between sublinear (and superlinear) maps defined in Dedekind complete vector lattices via Definition~\ref{D: DC sblin and splin fun cal} and vector lattice homomorphisms. Theorem~\ref{T: phi hom interchange} is central to this paper, as it will be employed several times throughout the manuscript.
	
	\begin{theorem}\label{T: phi hom interchange}
		Fix $n\in\N$. Let $E$ and $F$ be Dedekind complete vector lattices, let $T\colon E\to F$ be a vector lattice homomorphism, and put $f_1,\dots,f_n\in E$. 
		\begin{itemize}
			\item[(i)] If $\phi\colon\R^n\to\R$ is a sublinear map, then
			\[
			T\bigl(\phi(f_1,\dots,f_n)\bigr)=\phi\bigl(T(f_1),\dots,T(f_n)\bigr).
			\]
			\item[(ii)] If $\psi\colon\R^n\to\R$ is a superlinear map, then
			\[
			T\bigl(\psi(f_1,\dots,f_n)\bigr)=\psi\bigl(T(f_1),\dots,T(f_n)\bigr).
			\]
		\end{itemize}
	\end{theorem}
	
	\begin{proof}
		We only prove part (i), noting that part (ii) is handled similarly. To this end, assume that $\phi\colon\R^n\to\R$ is a sublinear map. By definition, we have
		\[
		\phi(f_1,\dots,f_n)=\underset{\bar{a}\in\underline{\partial}\phi(\bar{0})}{\sup}\sum_{i=1}^{n}a_if_i.
		\]
		Next, let $m\in\N$ be arbitrary. For each  $\bar{a}\in\underline{\partial}\phi(\bar{0})$, define
		\[
		O_{\bar{a}, \frac{1}{2mn}}:=\left\{\bar{x}\in\R^n\ :\ \max\bigl\{|x_i-a_i|\ :\ i\in\{1,\dots, n\}\bigr\}<\dfrac{1}{2mn}\right\}.
		\]
		Observe that $\mathcal{O}:=\left\{O_{\bar{a},\frac{1}{2mn}}\ :\ \bar{a}\in\underline{\partial}\phi(\bar{0})\right\}$ is an open cover of $\underline{\partial}\phi(\bar{0})$, and, since $\underline{\partial}\phi(\bar{0})$ is compact, $\mathcal{O}$ contains a finite subcover 
		\[
		\mathcal{O}_0:=\left\{O_{\overline{a^1},\frac{1}{2mn}},\dots,O_{\overline{a^k},\frac{1}{2mn}}\right\},
		\]
		with $\overline{a^1},\dots,\overline{a^k}\in\underline{\partial}\phi(\bar{0})$. Thus, fixing an arbitrary $\bar{a}\in\underline{\partial}\phi(\bar{0})$, there exists $j_0\in\{1,\dots,k\}$ for which $\bar{a}\in O_{\overline{a^{j_0}},\frac{1}{2mn}}$. Therefore, we have
		\[
			\max\bigl\{|a_i-a^{j_0}_i|\ :\ i\in\{1,\dots,n\}\bigr\}<\dfrac{1}{2mn}.	
		\]
	Hence we obtain
		\begin{align*}
			\sum_{i=1}^{n}a_if_i-\bigvee_{j=1}^k\sum_{i=1}^{n}a^j_if_i&\leq \sum_{i=1}^{n}a_if_i-\sum_{i=1}^{n}a^{j_0}_if_i\\
			&=\sum_{i=1}^{n}(a_i-a^{j_0}_i)f_i\\
			&\leq\sum_{i=1}^{n}|a_i-a^{j_0}_i||f_i|\\
			&\leq\dfrac{1}{2m}\sum_{i=1}^{n}|f_i|.
		\end{align*}
		It follows that
		\[
		\left|\phi(f_1,\dots,f_n)-\bigvee_{j=1}^k\sum_{i=1}^{n}a^j_if_i\right|=\phi(f_1,\dots,f_n)-\bigvee_{j=1}^k\sum_{i=1}^{n}a^j_if_i\leq\dfrac{1}{2m}\sum_{i=1}^{n}|f_i|.
		\]
		Similarly, we have
		\[
		\left|\phi\bigl(T(f_1),\dots,T(f_n)\bigr)-\bigvee_{j=1}^k\sum_{i=1}^{n}a^j_iT(f_i)\right|\leq\dfrac{1}{2m}\sum_{i=1}^{n}|T\bigl(f_i\bigr)|
		\]
		in $F$. Therefore, recalling that $T$ is a vector lattice homomorphism, we get
		\begin{align*}
			&\left|T\bigl(\phi(f_1,\dots,f_n)\bigr)-\phi\bigl( T(f_1),\dots,T(f_n)\bigr)\right|\\
			&\ \leq\left|T\bigl(\phi(f_1,\dots,f_n)\bigr)-T\left(\bigvee_{j=1}^k\sum_{i=1}^{n}a^j_if_i\right)\right|+\left|\phi\bigl( T(f_1),\dots,T(f_n)\bigr)-T\left(\bigvee_{j=1}^k\sum_{i=1}^{n}a^j_if_i\right)\right|\\
			&\ =\left|T\left(\phi(f_1,\dots,f_n)-\bigvee_{j=1}^k\sum_{i=1}^{n}a^j_if_i\right)\right|+\left|\phi\bigl( T(f_1),\dots,T(f_n)\bigr)-\bigvee_{j=1}^k\sum_{i=1}^{n}a^j_iT(f_i)\right|\\
			&\leq T\left(\dfrac{1}{2m}\sum_{i=1}^{n}|f_i|\right)+\dfrac{1}{2m}\sum_{i=1}^{n}|T\bigl(f_i\bigr)|\\
			&\ =\dfrac{1}{m}\sum_{i=1}^{n}T\bigl(|f_i|\bigr).
		\end{align*}
		Since $m\in\N$ was arbitrary and $F$ is Archimedean, we conclude that
		\[
		T\bigl(\phi(f_1,\dots,f_n)\bigr)=\phi\bigl( T(f_1),\dots,T(f_n)\bigr).
		\]
	\end{proof}

	When we are working with vector sublattices, the following notation often proves to be useful.
	
	\begin{notation}
		Given $n\in\N$, a nonempty subset $A\subseteq\R^n$, an Archimedean vector lattice $E$, a vector sublattice $E_0$ of $E$, and a function $f\colon A\to E_0$, we write $E_0-\underset{\bar{a}\in A}{\inf} f(\bar{a})$ when $\underset{\bar{a}\in A}{\inf} f(\bar{a})$ is taken in $E_0$ (and exists in $E_0$) and $E-\underset{\bar{a}\in A}{\inf} f(\bar{a})$ when $\underset{\bar{a}\in A}{\inf} f(\bar{a})$ is taken in $E$ (and exists in $E$).
	\end{notation}
	
	The final result of this section is a corollary to Theorem~\ref{T: phi hom interchange} and states that our functional calculus can be calculated in different vector sublattices without changing the result. The proof follows immediately from using Proposition~\ref{T: phi hom interchange} with respect to the inclusion map $T\colon E_0\to E$.
	
	\begin{corollary}\label{C: DC phi subspace invariant}
		Fix $n\in\N$, let $E$ be a Dedekind complete vector lattice, and suppose $E_0$ is a Dedekind complete vector sublattice of $E$. Put $f_1,\dots,f_n\in E_0$.
		\begin{itemize}
			\item[(i)] If $\phi\colon\R^n\to\R$ is a sublinear map, then
			\[
			E_0-\underset{\bar{a}\in\underline{\partial}\phi(\bar{0})}{\sup}\ \sum_{i=1}^{n}a_if_i=E-\underset{\bar{a}\in\underline{\partial}\phi(\bar{0})}{\sup}\ \sum_{i=1}^{n}a_if_i.
			\]
			\item[(ii)] If $\psi\colon\R^n\to\R$ is a superlinear map, then
			\[
			E_0-\underset{\bar{a}\in\overline{\partial}\psi(\bar{0})}{\inf}\ \sum_{i=1}^{n}a_if_i=E-\underset{\bar{a}\in\overline{\partial}\psi(\bar{0})}{\inf}\ \sum_{i=1}^{n}a_if_i.
			\]
		\end{itemize}
		In other words, the value of $\phi(f_1,\dots,f_n)$ ($\psi(f_1,\dots,f_n)$) remains the same whether its associated supremum (infimum) is calculated in $E_0$ or $E$. 
	\end{corollary}

	\section{An Archimedean Vector Lattice Functional Calculus for Sublinear and Superlinear Maps}\label{S: Arch phi fun cal}
	
	In this section we extend, when possible, the Dedekind complete functional calculus for sublinear and superlinear maps (Definition~\ref{D: DC sblin and splin fun cal}) to more general Archimedean vector lattices. We begin with a motivating example.
	
\begin{example}\label{E: smc not DC}
An Archimedean vector lattice $E$ is said to be  \textit{square mean closed} (see \cite[Secton~2.2]{Az}) if for every $f,g\in E$,
\[
\underset{\theta\in[0,2\pi]}{\sup}\{(\cos\theta)f+(\sin\theta)g\}
\]
exists in $E$. There exist Archimedean vector lattices which are square mean closed but not Dedekind complete; and example is the space $\text{Lip}[0,1]$ of all Lipschitz functions defined on $[0,1]$ (see \cite[Remark~4]{Az}). However, consider the sublinear map $s\colon\R^2\to\R$ defined by
\[
s(x,y):=\sqrt{x^2+y^2}.
\]
Let $E$ be any square mean closed Archimedean vector lattice, and fix $f,g\in E$. Recalling that $\underline{\partial}s(\bar{0})=\{(a,b)\in\R^2\ :\ a^2+b^2\leq 1\}$, it is readily proven that
\[
\underset{(a,b)\in \underline{\partial}s(\bar{0})}{\sup}(af+bg)
\]
exists in $E$ and that
\[
\underset{(a,b)\in \underline{\partial}s(\bar{0})}{\sup}(af+bg)=\underset{\theta\in[0,2\pi]}{\sup}\bigl((\cos\theta)f+(\sin\theta)g\bigr).
\]
Hence in particular we can define $s(f,g)$, for any $f,g\in \text{Lip}[0,1]$, in the same manner as in Definition~\ref{D: DC sblin and splin fun cal}, even though $\text{Lip}[0,1]$ is not Dedekind complete.
\end{example}

We expand on the idea presented in Example~\ref{E: smc not DC} in our next definition.

	\begin{definition}\label{D: sblin and splin fun cal}
		Let $E$ be an Archimedean vector lattice, fix $n\in\N$, and put $f_1,\dots,f_n\in E$.
		\begin{itemize}
			\item[(i)] Given a sublinear map $\phi\colon\R^n\to\R$, if $\underset{\bar{a}\in\underline{\partial}\phi(\bar{0})}{\sup}\ \sum_{i=1}^{n}a_if_i$ exists in $E$, then we say that $\phi(f_1,\dots,f_n)$ \textit{is defined in} $E$ and write
			\[
				\phi(f_1,\dots,f_n):=\underset{\bar{a}\in\underline{\partial}\phi(\bar{0})}{\sup}\ \sum_{i=1}^{n}a_if_i.
			\]
			\item[(ii)] For a superlinear map $\psi\colon\R^n\to\R$, if $\underset{\bar{a}\in\overline{\partial}\psi(\bar{0})}{\inf}\ \sum_{i=1}^{n}a_if_i$ exists in $E$, then we say that $\psi(f_1,\dots,f_n)$ \textit{is defined in} $E$ and write
			\[
				\psi(f_1,\dots,f_n):=\underset{\bar{a}\in\overline{\partial}\psi(\bar{0})}{\inf}\ \sum_{i=1}^{n}a_if_i.
			\]
		\end{itemize}
	\end{definition}
	
	It is known that Definition~\ref{D: sblin and splin fun cal} agrees with the continuous Archimedean vector lattice functional calculus for sublinear and superlinear maps in uniformly complete Archimedean vector lattices \cite[Theorem~5.5]{Kus}. However, we will not exploit this fact in this paper, as we aim to develop a functional calculus that is completely independent of the continuous Archimedean vector lattice function calculus of \cite{BusdPvR}.

By following the proof of Theorem~\ref{T: phi hom interchange}, we obtain the following more general result. Theorem~\ref{T: phi hom interchange avl} will primarily be utilized in the proof of Theorem~\ref{T: swapTandh} to verify that our semicontinuous Archimedean vector lattice functional calculus (to be defined in Definition~\ref{D: funcal}) properly extends the continuous Archimedean vector lattice functional calculus.
	
	\begin{theorem}\label{T: phi hom interchange avl}
		Fix $n\in\N$. Let $E$ and $F$ be Archimedean vector lattices, let $T\colon E\to F$ be a vector lattice homomorphism, and put $f_1,\dots,f_n\in E$. 
		\begin{itemize}
			\item[(i)] If $\phi\colon\R^n\to\R$ is a sublinear map, $\phi(f_1,\dots,f_n)$ is defined in $E$, and the expression $\phi\bigl(T(f_1),\dots,T(f_n)\bigr)$ is defined in $F$, then
			\[
			T\bigl(\phi(f_1,\dots,f_n)\bigr)=\phi\bigl(T(f_1),\dots,T(f_n)\bigr).
			\]
			\item[(ii)] If $\psi\colon\R^n\to\R$ is a superlinear map, $\psi(f_1,\dots,f_n)$ is defined in $E$, and the expression $\psi\bigl(T(f_1),\dots,T(f_n)\bigr)$ is defined in $F$, then
			\[
			T\bigl(\psi(f_1,\dots,f_n)\bigr)=\psi\bigl(T(f_1),\dots,T(f_n)\bigr).
			\]
		\end{itemize}
	\end{theorem}
	
	We utilize the following notation in our next corollary.
	
	\begin{notation}
		Given an Archimedean vector lattice $E$, we denote its Dedekind completion by $E^\delta$.
	\end{notation}
	
	The next result is a corollary to Theorem~\ref{T: phi hom interchange avl} and states that the semicontinuous Archimedean vector lattice functional calculus can also be calculated in different vector sublattices without changing the result.
	
	\begin{corollary}\label{C: subspace invariant}
		Let $n\in\N$, let $E$ be an Archimedean vector lattice, and suppose $E_0$ is a vector sublattice of $E$. Put $f_1,\dots,f_n\in E_0$.
		\begin{itemize}
			\item[(i)] Suppose $\phi\colon\R^n\to\R$ is a sublinear map. If $E_0-\underset{\bar{a}\in\underline{\partial}\phi(\bar{0})}{\sup}\ \sum_{i=1}^{n}a_if_i$ exists in $E_0$, then $E-\underset{\bar{a}\in\underline{\partial}\phi(\bar{0})}{\sup}\ \sum_{i=1}^{n}a_if_i$ exists in $E$, and
			\[
			E_0-\underset{\bar{a}\in\underline{\partial}\phi(\bar{0})}{\sup}\ \sum_{i=1}^{n}a_if_i=E-\underset{\bar{a}\in\underline{\partial}\phi(\bar{0})}{\sup}\ \sum_{i=1}^{n}a_if_i.
			\]
			\item[(ii)] Suppose $\psi\colon\R^n\to\R$ is a superlinear map. If $E_0-\underset{\bar{a}\in\overline{\partial}\psi(\bar{0})}{\inf}\ \sum_{i=1}^{n}a_if_i$ exists in $E_0$, then $E-\underset{\bar{a}\in\overline{\partial}\psi(\bar{0})}{\inf}\ \sum_{i=1}^{n}a_if_i$ exists in $E$, and
			\[
			E_0-\underset{\bar{a}\in\overline{\partial}\psi(\bar{0})}{\inf}\ \sum_{i=1}^{n}a_if_i=E-\underset{\bar{a}\in\overline{\partial}\psi(\bar{0})}{\inf}\ \sum_{i=1}^{n}a_if_i.
			\]
		\end{itemize}
		In other words, the value of $\phi(f_1,\dots,f_n)$ ($\psi(f_1,\dots,f_n)$) remains the same whether its associated supremum (infimum) is calculated in $E_0$ or $E$. 
	\end{corollary}
	
	\begin{proof}
		Assume that $E_0-\underset{\bar{a}\in\underline{\partial}\phi(\bar{0})}{\sup}\ \sum_{i=1}^{n}a_if_i$ exists in $E_0$. By Proposition~\ref{P: dc implies phic}, we have that $E^\delta-\underset{\bar{a}\in\underline{\partial}\phi(\bar{0})}{\sup}\ \sum_{i=1}^{n}a_if_i$ exists in $E^\delta$ as well. Thus, applying Theorem~\ref{T: phi hom interchange avl} with respect to the inclusion map $T\colon E_0\to E^\delta$, we get
		\[
		E_0-\underset{\bar{a}\in\underline{\partial}\phi(\bar{0})}{\sup}\ \sum_{i=1}^{n}a_if_i=E^\delta-\underset{\bar{a}\in\underline{\partial}\phi(\bar{0})}{\sup}\ \sum_{i=1}^{n}a_if_i.
		\]
		Then $E^\delta-\underset{\bar{a}\in\underline{\partial}\phi(\bar{0})}{\sup}\ \sum_{i=1}^{n}a_if_i\in E_0$, and therefore $E^\delta-\underset{\bar{a}\in\underline{\partial}\phi(\bar{0})}{\sup}\ \sum_{i=1}^{n}a_if_i\in E$ since $E_0$ is a vector sublattice of $E$. Hence $E-\underset{\bar{a}\in\underline{\partial}\phi(\bar{0})}{\sup}\ \sum_{i=1}^{n}a_if_i$ exists in $E$ and
		\[
		E-\underset{\bar{a}\in\underline{\partial}\phi(\bar{0})}{\sup}\ \sum_{i=1}^{n}a_if_i=E^\delta-\underset{\bar{a}\in\underline{\partial}\phi(\bar{0})}{\sup}\ \sum_{i=1}^{n}a_if_i=E_0-\underset{\bar{a}\in\underline{\partial}\phi(\bar{0})}{\sup}\ \sum_{i=1}^{n}a_if_i.
		\]
	\end{proof}
	
	Of particular interest in the study of the sublinear and superlinear Archimedean vector lattice functional calculus is vector lattices in which this functional calculus, with respect to a specific sublinear or superlinear map $\xi\colon\R^n\to\R$, can always be performed. This leads us to our next definition.

\begin{definition}\label{D: xi closed}
	Fix $n\in\N$. Let $\xi\colon\R^n\to\R$ be a sublinear map or a superlinear map. We call an Archimedean vector lattice $E$ $\xi$-\textit{closed} if for all $f_1,\dots,f_n\in E$, $\xi(f_1,\dots,f_n)$ is defined in $E$.
\end{definition}

In light of Definition~\ref{D: xi closed}, Proposition~\ref{P: dc implies phic} can be rephrased as stating that every Dedekind complete vector lattice is $\xi$-closed for every sublinear or superlinear map $\xi\colon\R^n\to\R$.

Given any Archimedean vector lattice $E$ and a sublinear or superlinear map $\xi\colon\R^n\to\R$, there exists a smallest $\xi$-closed Archimedean vector lattice that contains $E$ as a vector sublattice. Indeed, following Azouzi's construction of the square mean closure in \cite[Remark~4]{Az}, we define the $\xi$-closure of an Archimedean vector lattice $E$.

\begin{definition}\label{D: xi-closure}
	Given $n\in\N$, an Archimedean vector lattice $E$, and a sublinear or superlinear map $\xi\colon\R^n\to\R$, we set $E_1:=E$. For $m\in\N$, we define $E_{m+1}$ to be the vector sublattice of $E^\delta$ generated by the set
	\[
	\bigl\{E_m\cup\{\xi(f_1,\dots,f_n)\ :\ f_1,\dots,f_n\in E_n\}\bigr\}.
	\]
	We define $E^\xi:=\bigcup_{m\in\N}E_m$ and call $E^\xi$ the $\xi$-\textit{closure} of $E$.
\end{definition}

We conclude this section with the following proposition, which is immediate.

\begin{proposition}
	If $E$ is an Archimedean vector lattice, $n\in\N$, and $\xi\colon\R^n\to\R$ is a sublinear or superlinear map, then $E^\xi$ is an $\xi$-closed Archimedean vector lattice.
\end{proposition}

	\section{A Dedekind Complete Vector Lattice Functional Calculus for Semicontinuous Positively Homogeneous Functions}\label{S: h fun cal}
	
	We extend the functional calculus for sublinear and superlinear maps defined in Dedekind complete vector lattices via Definition~\ref{D: DC sblin and splin fun cal} to semicontinuous positively homogeneous functions that are bounded on the unit sphere in this section. In Section~\ref{S: non DC h fun cal}, we will extend, when possible, the theory of this section to a wider class of Archimedean vector lattices.
	
	\begin{notation}
		For $n\in\N$, we as usual denote the unit sphere of $\R^n$ by $S^{n-1}$.
	\end{notation}
	
	The following definition is standard, but we include it for completeness.
	
	\begin{definition}\label{D: bdd on unit sphere}
		Given $n\in\N$ and a function $f\colon\R^n\to\R$, we say $f$ is \textit{bounded on the unit sphere} (respectively, \textit{bounded on the closed unit ball}) if there exists $m,M\in\R$ such that for all $\bar{x}\in S^{n-1}$ (respectively, for all $\bar{x}\in B_1[\bar{0}]:=\{\bar{x}\in\R^n\ :\ \|\bar{x}\|\leq 1\}$), we have $m\leq f(\bar{x})\leq M$.
	\end{definition}
	
	\begin{remark}\label{R: bdd on unit ball iff bdd on unit sphere}
		For $n\in\N$, we note that a positively homogeneous function $h\colon\R^n\to\R$ is bounded on the unit sphere if and only if it is bounded on the closed unit ball. Indeed, if $h$ is bounded on the closed unit ball, then $h$ is clearly bounded on the unit sphere. Conversely, if $h$ is bounded on the unit sphere, then there exists $m,M\in(0,\infty)$, for which $-m\leq h(\bar{x})\leq M$ holds for every $\bar{x}\in S^{n-1}$. If $\bar{x}=\bar{0}$, then $-m\leq h(\bar{x})=0\leq M$. On the other hand, if $\bar{x}\in B_1[\bar{0}]$ and $\bar{x}\neq 0$, then $\dfrac{\bar{x}}{\|\bar{x}\|}\in S^{n-1}$ and $0<\|\bar{x}\|\leq 1$. We thus get
		\[
		-m\leq -m\|x\|\leq h\left(\dfrac{\bar{x}}{\|\bar{x}\|}\right)\|\bar{x}\|\leq M\|x\|\leq M.
		\]
		
		Since $h(\bar{x})=h\left(\dfrac{\bar{x}}{\|\bar{x}\|}\right)\|\bar{x}\|$, we see that $h$ is bounded on the closed unit ball.
	\end{remark}
	
	The following result constitutes the primary tool we utilize for the development of the semicontinuous Archimedean vector lattice functional calculus.
	
	\begin{proposition}\cite[Lemma~5.2]{DR}\label{P: DR} 
		If $n\in\N$ and $h\colon\R^n\to\R$ is a positively homogeneous function that is bounded on the unit sphere, then the following hold.
		\begin{itemize}
			\item[(i)] $h$ is upper semicontinuous (on $\R^n$) if and only if there exists a family $\Phi$ of real-valued sublinear maps defined on $\R^n$ such that
			\begin{equation}\label{eq: inf-sublinear rep}
				h(\bar{x})=\underset{\phi\in\Phi}{\inf}\ \phi(\bar{x})\qquad (\bar{x}\in\R^n).
			\end{equation}
			\item[(ii)] $h$ is lower semicontinuous (on $\R^n$) if and only if there exists a family $\Psi$ of real-valued superlinear maps defined on $\R^n$ such that
			\begin{equation}\label{eq: sup-superlinear rep}
				h(\bar{x})=\underset{\psi\in\Psi}{\sup}\ \psi(\bar{x})\qquad (\bar{x}\in\R^n).
			\end{equation}
			\item[(iii)] $h$ is continuous (on $\R^n$) if and only if there exists a family $\Phi$ of real-valued sublinear maps defined on $\R^n$ and a collection $\Psi$ of real-valued superlinear maps defined on $\R^n$ for which the identities \eqref{eq: inf-sublinear rep} and \eqref{eq: sup-superlinear rep} hold.
		\end{itemize}
	\end{proposition}
	
	We note here that \cite[Lemma~5.2]{DR} is stated for positively homogeneous functions which are bounded on the closed unit ball. By Remark~\ref{R: bdd on unit ball iff bdd on unit sphere} however, it is sufficient to consider positively homogeneous functions that are bounded on the unit sphere. 
	
	\begin{definition}\label{D: inf-sublinear and sup-superlinear}
		In the setting of Proposition~\ref{P: DR}, we refer to \eqref{eq: inf-sublinear rep} as an \textit{inf-sublinear} representation for $h$, where as \eqref{eq: sup-superlinear rep} will be called a \textit{sup-superlinear} representation for $h$ in this paper.
	\end{definition}
	
	In order to prove that our semicontinuous Dedekind complete vector lattice functional calculus can be defined in any Dedekind complete vector lattice, we need the following definition.
	
	\begin{definition}\label{D: dom}
		Given $n\in\N$ and a function $f\colon\R^n\to\R$, we say that $f$ \textit{dominates a superlinear map} if there exists a superlinear map $\psi\colon\R^n\to\R$ such that, for all $\bar{x}\in\R^n$, we have $f(\bar{x})\geq\psi(\bar{x})$. Likewise, we say that $f$ is \textit{dominated by a sublinear map} if there exists a sublinear map $\phi\colon\R^n\to\R$ such that, for all $\bar{x}\in\R^n$, we have $f(\bar{x})\leq\phi(\bar{x})$.  
	\end{definition}
	
	The following lemma will also be needed.
	
	\begin{lemma}\label{L: bdd on B iff dominating}
		Fix $n\in\N$, and let $h\colon\R^n\to\R$ be a positively homogeneous function. The following are equivalent.
		\begin{itemize}
			\item[(i)] $h$ is bounded on the unit sphere, and
			\item[(ii)] $h$ dominates a superlinear map, and $h$ is dominated by a sublinear map.
		\end{itemize}
	\end{lemma}
	
	\begin{proof}
		$(i)\implies(ii)$ Suppose $h$ is bounded on the unit sphere. As noted in Remark~\ref{R: bdd on unit ball iff bdd on unit sphere}, there exist $m,M\in(0,\infty)$ such that for all $\bar{x}\in S^{n-1}$, $-m\leq h(\bar{x})\leq M$. Define $\psi,\phi\colon\R^n\to\R$ by
		\[
		\psi(\bar{x})=-m\|\bar{x}\|\quad (\bar{x}\in\R^n)\qquad \text{and}\qquad \phi(\bar{x})=M\|x\|\quad (\bar{x}\in\R^n).
		\]
		It is clearly seen that $\psi$ is superlinear and $\phi$ is sublinear.
		
		Let $\bar{x}\in\R^n$. If $\bar{x}=\bar{0}$, then $\psi(\bar{x})=h(\bar{x})=\phi(\bar{x})=0$, since $\psi$, $h$, and $\phi$ are all positively homogeneous. Suppose that $\bar{x}\neq\bar{0}$, so that $\dfrac{x}{\|\bar{x}\|}\in S^{n-1}$. We thus obtain
		\[
		\psi(\bar{x})=-m\|\bar{x}\|\leq h\left(\dfrac{\bar{x}}{\|\bar{x}\|}\right)\|\bar{x}\|\leq M\|\bar{x}\|=\phi(\bar{x}).
		\]
		Since $h\left(\dfrac{\bar{x}}{\|\bar{x}\|}\right)\|\bar{x}\|=h(\bar{x})$, we have the desired implication.
		
		$(ii)\implies(i)$ This implication immediately follows from the fact that superlinear maps and sublinear maps are continuous on $\R^n$ and are hence bounded on the compact unit sphere.
	\end{proof}
	
	We next prove that if $h\colon\R^n\to\R$ is an upper (lower) semicontinuous positively homogeneous function with inf-sublinear representation $h(\bar{x})=\underset{\phi\in\Phi}{\inf}\ \phi(\bar{x})\quad (\bar{x}\in\R^n)$ $\bigl($sup-superlinear representation $h(\bar{x})=\underset{\psi\in\Psi}{\sup}\ \psi(\bar{x})\quad (\bar{x}\in\R^n)$$\bigr)$, $E$ is a Dedekind complete vector lattice, and $f_1,\dots,f_n\in E$, then $\underset{\phi\in\Phi}{\inf}\ \phi(f_1,\dots,f_n)\quad (\underset{\psi\in\Psi}{\sup}\ \psi(f_1,\dots,f_n))$ exists in $E$. This initial step is necessary to define our semicontinuous Dedekind complete vector lattice functional calculus in Definition~\ref{D: DC h funcal}.
	
	\begin{theorem}\label{T: dc implies hc}
		Fix $n\in\N$. Let $h\colon\R^n\to\R$ be a positively homogeneous function that is bounded on the unit sphere, suppose that $E$ is a Dedekind complete vector lattice, and put $f_1,\dots,f_n\in E$.
		\begin{itemize}
			\item[(i)] If $h$ is upper semicontinuous with inf-sublinear representation 
			\[
			h(\bar{x})=\underset{\phi\in\Phi}{\inf}\ \phi(\bar{x})\quad (\bar{x}\in\R^n),
			\]
			then $\underset{\phi\in\Phi}{\inf}\ \phi(f_1,\dots,f_n)$ exists in $E$.
			\item[(ii)] If $h$ is lower semicontinuous with sup-superlinear representation 
			\[
			h(\bar{x})=\underset{\psi\in\Psi}{\sup}\ \psi(\bar{x})\quad (\bar{x}\in\R^n),
			\]
			then $\underset{\psi\in\Psi}{\sup}\ \psi(f_1,\dots,f_n)$ exists in $E$.
		\end{itemize}
	\end{theorem}
	
	\begin{proof}
		(i) Assume that $h\colon\R^n\to\R$ is upper semicontinuous. Let $h(\bar{x})=\underset{\phi\in\Phi}{\inf}\phi(\bar{x})\quad (\bar{x}\in\R^n)$ be an inf-sublinear representation for $h$, and fix $\phi\in\Phi$. By Lemma~\ref{L: bdd on B iff dominating}, $h$ dominates a superlinear map $\psi\colon\R^n\to\R$. Let $E_0$ be the vector sublattice of $E$ generated by
		\[
		\{f_1,\dots,f_n, \phi(f_1,\dots,f_n), \psi(f_1,\dots,f_n)\},
		\]
		and let $\omega\colon E_0\to\R$ be a vector lattice homomorphism. Using Proposition~\ref{T: phi hom interchange} in both of the equalities below, we obtain
		\begin{align*}
			\omega\bigl(\phi(f_1,\dots,f_n)\bigr)&=\phi\bigl(\omega(f_1),\dots,\omega(f_n)\bigr)\\
			&\geq h\bigl(\omega(f_1),\dots,\omega(f_n)\bigr)\\
			&\geq \psi\bigl(\omega(f_1),\dots,\omega(f_n)\bigr)\\
			&=\omega\bigl(\psi(f_1,\dots,f_n)\bigr).
		\end{align*}
		Since the set of all vector lattice homomorphisms $\omega\colon E_0\to\R$ separate the points of $E_0$ \cite[remark 1.2(ii), Theorem~2.2]{BusvR}, we conclude that $\phi(f_1,\dots,f_n)\geq\psi(f_1,\dots,f_n)$. This inequality holds for any $\phi\in\Phi$, and $E$ is Dedekind complete, so we have that $\underset{\phi\in\Phi}{\inf}\ \phi(f_1,\dots,f_n)$ exists in $E$.
		
		(ii) A symmetric proof handles the case in which $h$ is lower semicontinuous.
	\end{proof}

	Next we introduce our Dedekind complete vector lattice functional calculus for semicontinuous positively homogeneous functions that are bounded on the unit sphere.
	
	\begin{definition}\label{D: DC h funcal}
		Let $n\in\N$, and assume that $h\colon\R^n\to\R$ is an upper semicontinuous positively homogeneous function that that is bounded on the unit sphere. Furthermore, let $\Phi$ be a family of sublinear maps such that for each $\bar{x}\in\R^n$, we have $h(\bar{x})=\underset{\phi\in\Phi}{\inf}\ \phi(\bar{x})$.
		Next consider a Dedekind complete vector lattice $E$, and put $f_1,\dots,f_n\in E$. Noting that for every $\phi\in\Phi$, $\phi(f_1,\dots,f_n)$ is defined in $E$ (Proposition~\ref{P: dc implies phic}) and $\underset{\phi\in\Phi}{\inf}\ \phi(f_1,\dots,f_n)$ exists in $E$ (Theorem~\ref{T: dc implies hc}), we define
		\[
		h(f_1,\dots,f_n):=\underset{\phi\in\Phi}{\inf}\ \phi(f_1,\dots,f_n).
		\]
		
		Likewise, consider a lower semicontinuous positively homogeneous function $h\colon\R^n\to\R$ that is bounded on the unit sphere. Let $\Psi$ be a family of real-valued superlinear maps defined on $\R^n$ such that for all $\bar{x}\in\R^n$, $h(\bar{x})=\underset{\psi\in\Psi}{\sup}\ \psi(\bar{x})$.
		Recalling that for every $\psi\in\Psi$, $\psi(f_1,\dots,f_n)$ is defined in $E$ (Proposition~\ref{P: dc implies phic}) and $\underset{\psi\in\Psi}{\sup}\ \psi(f_1,\dots,f_n)$ exists in $E$ (Theorem~\ref{T: dc implies hc}), we define
		\[
		h(f_1,\dots,f_n):=\underset{\psi\in\Psi}{\sup}\ \psi(f_1,\dots,f_n).
		\]
		We refer to this functional calculus as the \textit{semicontinuous Dedekind complete vector lattice functional calculus}.
	\end{definition}
	
	Our first task regarding the semicontinuous Dedekind complete vector lattice functional calculus from Definition~\ref{D: DC h funcal} is to show that it is well-defined. The following proposition states that two inf-sublinear (sup-superlinear) representations of an upper (lower) semicontinuous positively homogeneous function which is bounded on the unit sphere yield the same output in the semicontinuous Dedekind complete vector lattice functional calculus.
	
	\begin{proposition}\label{P: welldefined}
		Let $n\in\N$, and let $h\colon\R^n\to\R$ be a positively homogeneous function that is bounded on the unit sphere. Also let $E$ be a Dedekind complete vector lattice, and put $f_1,\dots,f_n\in E$.
		\begin{itemize}
			\item[(i)] If $h$ is upper semicontinuous and $\Phi$ and $\Xi$ are families of real-valued sublinear maps defined on $\R^n$ such that for all $\bar{x}\in\R^n$,
			\[
			h(\bar{x})=\underset{\phi\in\Phi}{\inf}\ \phi(\bar{x})=\underset{\xi\in\Xi}{\inf}\ \xi(\bar{x}),
			\]
			then
			\[
			\underset{\phi\in\Phi}{\inf}\ \phi(f_1,\dots,f_n)=\underset{\xi\in\Xi}{\inf}\ \xi(f_1,\dots,f_n).
			\]
			\item[(ii)] If $h$ is lower semicontinuous, and $\Psi$ and $\Xi$ are families of real-valued superlinear maps defined on $\R^n$ such that for all $\bar{x}\in\R^n$,
			\[
			h(\bar{x})=\underset{\psi\in\Psi}{\inf}\ \psi(\bar{x})=\underset{\xi\in\Xi}{\inf}\ \xi(\bar{x}),
			\]
			then
			\[
			\underset{\psi\in\Psi}{\sup}\ \psi(f_1,\dots,f_n)=\underset{\xi\in\Xi}{\sup}\ \xi(f_1,\dots,f_n).
			\]
		\end{itemize}
	\end{proposition}
	
	\begin{proof}
		We only prove part (i) of the proposition, noting that the proof of part (ii) is similar. To this end, suppose $h$ is upper semicontinuous and has inf-sublinear representations
		\[	h(\bar{x})=\underset{\phi\in\Phi}{\inf}\ \phi(\bar{x})\quad \text{and}\quad h(\bar{x})=\underset{\xi\in\Xi}{\inf}\ \xi(\bar{x})\quad (\bar{x}\in\R^n).
		\]
		Fix $\phi_0\in\Psi$ and $\xi_0\in\Xi$, and let $E_0$ be the vector sublattice of $E$ generated by
		\[
		\left\{f_1,\dots,f_n, \phi_0(f_1,\dots,f_n), \xi_0(f_1,\dots,f_n), \underset{\phi\in\Phi}{\inf}\ \phi(f_1,\dots,f_n), \underset{\xi\in\Xi}{\inf}\ \xi(f_1,\dots,f_n)\right\}.
		\]
		Next let $\omega\colon E_0\to\R$ be a vector lattice homomorphism. Using Proposition~\ref{T: phi hom interchange} in the equality below, we have
		\[
		\omega\left(\underset{\phi\in\Phi}{\inf}\ \phi(f_1,\dots,f_n)\right)\leq \omega\bigl(\phi_0(f_1,\dots,f_n)\bigr)=\phi_0\bigl(\omega(f_1),\dots,\omega(f_n)\bigr).
		\]
		It follows that
		\[
		\omega\left(\underset{\phi\in\Phi}{\inf}\ \phi(f_1,\dots,f_n)\right)\leq \underset{\phi\in\Phi}{\inf}\ \phi\bigl(\omega(f_1),\dots,\omega(f_n)\bigr).
		\]
		Therefore, utilizing Proposition~\ref{T: phi hom interchange} again in the second equality below, we obtain
		\begin{align*}
			\omega\left(\underset{\phi\in\Phi}{\inf}\ \phi(f_1,\dots,f_n)\right)&\leq \underset{\phi\in\Phi}{\inf}\ \phi\bigl(\omega(f_1),\dots,\omega(f_n)\bigr)\\
			&=\underset{\xi\in\Xi}{\inf}\ \xi\bigl(\omega(f_1),\dots,\omega(f_n)\bigr)\\
			&\leq \xi_0\bigl(\omega(f_1),\dots,\omega(f_n)\bigr)\\
			&=\omega\bigl(\xi_0(f_1,\dots,f_n)\bigr).
		\end{align*}
		Since the set of all vector lattice homomorphisms $\omega\colon E_0\to\R$ separate the points of $E_0$ \cite[remark 1.2(ii), Theorem~2.2]{BusvR}, we have that
		\[
		\underset{\phi\in\Phi}{\inf}\ \phi(f_1,\dots,f_n)\leq \xi_0(f_1,\dots,f_n).
		\]
		Given that $\xi_0\in\Xi$ was arbitrary, it follows that $\underset{\phi\in\Phi}{\inf}\ \phi(f_1,\dots,f_n)\leq \underset{\xi\in\Xi}{\inf}\ \xi_0(f_1,\dots,f_n)$. However, a symmetrical argument shows that $\underset{\xi\in\Xi}{\inf}\ \xi_0(f_1,\dots,f_n)\leq \underset{\phi\in\Phi}{\inf}\ \phi(f_1,\dots,f_n)$. We thus obtain $\underset{\phi\in\Phi}{\inf}\ \phi(f_1,\dots,f_n)=\underset{\xi\in\Xi}{\inf}\ \xi_0(f_1,\dots,f_n)$. 
	\end{proof}
	
	By Proposition~\ref{P: DR}, a continuous positively homogeneous function $h\colon\R^n\to\R$, being upper semicontinuous, lower semicontinuous, and bounded on the unit sphere, has both an inf-sublinear representation and a sup-superlinear representation. This fact raises the question of well-definedness regarding the semicontinuous Dedekind complete vector lattice functional calculus applied to continuous positively homogeneous functions. Our next proposition settles this issue.
	
	\begin{proposition}\label{P: cts well defined}
		Fix $n\in\N$. Let $h\colon\R^n\to\R$ be a continuous positively homogeneous function, and suppose that
		\begin{equation}\label{eq: DC ex1 cts}
			h(\bar{x})=\underset{\phi\in\Phi}{\inf}\ \phi(\bar{x})\qquad (\bar{x}\in\R^n)
		\end{equation}
		and
		\begin{equation}\label{eq: DC ex2 cts}
			h(\bar{x})=\underset{\psi\in\Psi}{\sup}\ \psi(\bar{x})\qquad (\bar{x}\in\R^n)
		\end{equation}
		hold. If $E$ is a Dedekind complete vector lattice and $f_1,\dots,f_n\in E$, then
		\[
		\underset{\phi\in\Phi}{\inf}\ \phi(f_1,\dots,f_n) = \underset{\psi\in\Psi}{\sup}\ \psi(f_1,\dots,f_n).
		\]
	\end{proposition}
	
	\begin{proof}
		Let $E$ be a Dedekind complete vector lattice, and put $f_1,\dots,f_n\in E$. Fix $\phi_0\in\Phi, \psi_0\in\Psi$. From \eqref{eq: DC ex1 cts} and \eqref{eq: DC ex2 cts}, we have for any $\bar{x}\in\R^n$,
		\[
		\psi_0(\bar{x})\leq h(\bar{x})\leq\phi_0(\bar{x}).
		\]
		Let $E_0$ be the vector sublattice of $E$ generated by $\{f_1,\dots, f_n, \psi_0(f_1,\dots, f_n), \phi_0(f_1,\dots, f_n)\}$, and let $\omega\colon E_0\to\mathbb{R}$ be a vector lattice homomorphism. Using Theorem~\ref{T: phi hom interchange} in both identities below, we have
		\begin{align*}
			\omega\bigl(\psi_0(f_1,\dots, f_n)\bigr)&=\psi_0\Bigl(\omega(f_1),\dots, \omega(f_n)\Bigr)\\
			&\leq\phi_0\Bigl(\omega(f_1),\dots, \omega(f_n)\Bigr)\\
			&=\omega\bigl(\phi_0(f_1,\dots, f_n)\bigr).
		\end{align*}
		Since the set of all vector lattice homomorphisms $\omega\colon E_0\to\R$ separate the points of $E_0$ \cite[remark 1.2(ii), Theorem~2.2]{BusvR}, we have that $\psi_0(f_1,\dots, f_n)\leq\phi_0(f_1,\dots, f_n)$. It follows that
		\begin{equation}\label{eq: inf>sup}
			\underset{\phi\in\Phi}{\inf}\ \phi(f_1,\dots,f_n) \geq \underset{\psi\in\Psi}{\sup}\ \psi(f_1,\dots,f_n).
		\end{equation}
		
		Next fix $\epsilon>0$, and let $i\in\{1,\dots,n\}$. Using both of the above expressions \eqref{eq: DC ex1 cts} and \eqref{eq: DC ex2 cts} with $\bar{x}=\overline{e^{(i)}}$, we get
		\[
		h\left(\overline{e^{(i)}}\right)=\underset{\phi\in\Phi}{\inf}\ \phi\left(\overline{e^{(i)}}\right)=\underset{\psi\in\Psi}{\sup}\ \psi\left(\overline{e^{(i)}}\right)=\underset{\phi\in\Phi}{\inf}\ \underset{\bar{a}\in\underline{\partial}\phi(\bar{0})}{\sup}a_i=\underset{\psi\in\Psi}{\sup}\ \underset{\bar{b}\in\overline{\partial}\psi(\bar{0})}{\inf}b_i.
		\]
		Ergo, there exists $\phi_0\in\Phi$ and $\psi_0\in\Psi$ such that
		\[
		0\leq\underset{\bar{a}\in\underline{\partial}\phi_0(\bar{0})}{\sup}a_i - \underset{\bar{b}\in\overline{\partial}\psi_0(\bar{0})}{\inf}b_i\leq\dfrac{\epsilon}{n}.
		\]
		Also note that
		\begin{align*}
			\underset{\bar{a}\in\underline{\partial}\phi_0(\bar{0})}{\sup}a_i - \underset{\bar{b}\in\overline{\partial}\psi_0(\bar{0})}{\inf}b_i&=\underset{\bar{a}\in\underline{\partial}\phi_0(\bar{0})}{\sup}a_i + \underset{\bar{b}\in\overline{\partial}\psi_0(\bar{0})}{\sup}(-b_i)\\
			&=\underset{\bar{a}\in\underline{\partial}\phi_0(\bar{0}), \bar{b}\in\overline{\partial}\psi_0(\bar{0})}{\sup}(a_i-b_i),
		\end{align*}
		and so we have
		\[
		0\leq \underset{\bar{a}\in\underline{\partial}\phi_0(\bar{0}), \bar{b}\in\overline{\partial}\psi_0(\bar{0})}{\sup}(a_i-b_i)\leq\dfrac{\epsilon}{n}.
		\]
		It follows that
		\begin{align*}
			\left|\underset{\bar{a}\in\underline{\partial}\phi_0(\bar{0})}{\sup}\sum_{i=1}^{n}a_if_i - \underset{\bar{b}\in\overline{\partial}\psi_0(\bar{0})}{\inf}\sum_{i=1}^{n}b_if_i\right|&=\left|\underset{\bar{a}\in\underline{\partial}\phi_0(\bar{0})}{\sup}\sum_{i=1}^{n}a_if_i + \underset{\bar{b}\in\overline{\partial}\psi_0(\bar{0})}{\sup}\sum_{i=1}^{n}(-b_i)f_i\right|\\
			&=\left|\underset{\bar{a}\in\underline{\partial}\phi_0(\bar{0}), \bar{b}\in\overline{\partial}\psi_0(\bar{0})}{\sup}\sum_{i=1}^{n}(a_i-b_i)f_i\right|\\
			&\leq\sum_{i=1}^{n}\left(\underset{\bar{a}\in\underline{\partial}\phi_0(\bar{0}), \bar{b}\in\overline{\partial}\psi_0(\bar{0})}{\sup}(a_i-b_i)\right)|f_i|\\
			&\leq\epsilon\left(\sum_{i=1}^{n}|f_i|\right).
		\end{align*}
		Thus, using \eqref{eq: inf>sup} in the first identity below, we have
		\begin{align*}
			\left|\underset{\phi\in\Phi}{\inf}\ \phi(f_1,\dots,f_n)-\underset{\psi\in\Psi}{\sup}\ \psi(f_1,\dots,f_n)\right|&=\underset{\phi\in\Phi}{\inf}\ \phi(f_1,\dots,f_n)-\underset{\psi\in\Psi}{\sup}\ \psi(f_1,\dots,f_n)\\
			&=\underset{\phi\in\Phi}{\inf}\ \underset{\bar{a}\in\underline{\partial}\phi(\bar{0})}{\sup}\ \sum_{i=1}^{n}a_if_i - \underset{\psi\in\Psi}{\sup}\ \underset{\bar{b}\in\overline{\partial}\psi(\bar{0})}{\inf}\ \sum_{i=1}^{n}b_if_i\\
			&\leq \underset{\bar{a}\in\underline{\partial}\phi_0(\bar{0})}{\sup}\ \sum_{i=1}^{n}a_if_i - \underset{\bar{b}\in\overline{\partial}\psi_0(\bar{0})}{\inf}\ \sum_{i=1}^{n}b_if_i\\
			&\leq\left|\underset{\bar{a}\in\underline{\partial}\phi_0(\bar{0})}{\sup}\ \sum_{i=1}^{n}a_if_i - \underset{\bar{b}\in\overline{\partial}\psi_0(\bar{0})}{\inf}\ \sum_{i=1}^{n}b_if_i\right|\\
			&\leq\epsilon\left(\sum_{i=1}^{n}|f_i|\right).
		\end{align*}
		The proposition now follows, since $E$ is Archimedean.
	\end{proof}
	
	Another issue of well-definedness for the semicontinuous Dedekind complete vector lattice calculus concerns the relevant calculations in different Dedekind complete vector lattices, where one vector lattice is a vector sublattice of the other. Specifically, suppose that $n\in\N$, $h\colon\R^n\to\R$ is a semicontinuous positively homogeneous function that is bounded on the unit sphere, $E$ is a Dedekind complete vector lattice, $E_0$ is a Dedekind complete vector sublattice of $E$, and $f_1,\dots,f_n\in E_0$. Our next result states that the value of $h(f_1,\dots,f_n)$ remains the same whether the semicontinuous Dedekind complete vector lattice functional calculus is performed in $E_0$ or $E$.

	\begin{proposition}\label{P: DC subspace well-defined} Fix $n\in\N$, let $E$ be a Dedekind complete vector lattice, and suppose $E_0$ is a Dedekind complete vector sublattice of $E$. Put $f_1,\dots,f_n\in E_0$. Suppose $h\colon\R^n\to\R$ is a positively homogeneous function which is bounded on the unit sphere.
		\begin{itemize}
			\item[(i)] If $h$ is upper semicontinuous with inf-sublinear representation 
			\[
			h(\bar{x})=\underset{\phi\in\Phi}{\inf}\ \phi(\bar{x})\quad (\bar{x}\in\R^n),
			\]
			then
			\[
			E_0-\underset{\phi\in\Phi}{\inf}\ \phi(f_1,\dots,f_n)=E-\underset{\phi\in\Phi}{\inf}\ \phi(f_1,\dots,f_n).
			\]
			\item[(ii)] If $h$ is lower semicontinuous with sup-superlinear representation 
			\[
			h(\bar{x})=\underset{\psi\in\Psi}{\sup}\ \psi(\bar{x})\quad (\bar{x}\in\R^n),
			\]
			then
			\[
			E_0-\underset{\psi\in\Psi}{\sup}\ \psi(f_1,\dots,f_n)=E-\underset{\psi\in\Psi}{\sup}\ \psi(f_1,\dots,f_n).
			\]
		\end{itemize}
		In summary, the value of $h(f_1,\dots,f_n)$ remains the same whether its associated infimum (supremum) in its inf-sublinear (sup-superlinear) representation is calculated in $E_0$ or $E$. 
	\end{proposition}
	
	\begin{proof}
		We only prove (i), noting that (ii) is handled in a symmetrical manner. To this end, fix $\phi_0\in\Phi$. In $E_0$ we have $E_0-\underset{\phi\in\Phi}{\inf}\ \phi(f_1,\dots,f_n)\leq \phi_0(f_1,\dots,f_n)$, and we add that, by Corollary~\ref{C: DC phi subspace invariant}, the value of $\phi_0(f_1,\dots,f_n)$ remains the same wither its associated supremum is calculated in $E_0$ or $E$. Thus we obtain
		$E_0-\underset{\phi\in\Phi}{\inf}\ \phi(f_1,\dots,f_n)\leq E-\underset{\phi\in\Phi}{\inf}\ \phi(f_1,\dots,f_n)$. On the other hand, the fact that $E_0$ is a vector sublattice of $E$ implies that $E_0-\underset{\phi\in\Phi}{\inf}\ \phi(f_1,\dots,f_n)\geq E-\underset{\phi\in\Phi}{\inf}\ \phi(f_1,\dots,f_n)$. We thus have
		\[
		E_0-\underset{\phi\in\Phi}{\inf}\ \phi(f_1,\dots,f_n)= E-\underset{\phi\in\Phi}{\inf}\ \phi(f_1,\dots,f_n).
		\]
	\end{proof}
	
	In conclusion of this section, we have introduced the semicontinuous Dedekind complete vector lattice functional calculus and proven that it is well-defined.
	
	\section{An Archimedean Vector Lattice Functional Calculus for Semicontinuous Positively Homogeneous Functions}\label{S: non DC h fun cal}
	
	In this section we extend, when realizable, the semicontinuous Dedekind complete vector lattice functional calculus from Section~\ref{S: h fun cal} to more general Archimedean vector lattices. We begin with an example which illustrates that, while Dedekind completeness is a sufficient condition for us to execute our functional calculus in an Archimedean vector lattice, it is not a necessary condition.
	
	\begin{example}\label{E: step fcns}
		Consider the Archimedean vector lattice $S[0,1]$ of all real-valued step functions defined on $[0,1]$. For any $n\in\N$, $f_1,\dots,f_n\in S[0,1]$, and $k\in\{1,\dots,n\}$, we can express $f_k$ in the form $f_k(t)=\sum_{i=1}^{n}\lambda_{k,i}\chi_{A_i}(t)\quad \bigl(t\in[0,1]\bigr)$, for some $\lambda_{k,1},\dots,\lambda_{k,n}\in\R$. Given any upper (or lower) semicontinuous positively homogeneous function $h\colon\R^n\to\R$ bounded on the unit sphere, it is readily checked that $h(f_1,\dots,f_n)$ is defined in $S[0,1]$ and 
		\[
		\bigl(h(f_1,\dots,f_n)\bigr)(t)=\sum_{i=1}^{n}h(\lambda_{1,i},\dots,\lambda_{n,i})\chi_{A_i}(t)
		\]
		for each $t\in[0,1]$. However, $S[0,1]$ is not Dedekind complete. (It is not even uniformly complete.)
	\end{example}
	
	Using Example~\ref{E: step fcns} as motivation, we expand our semicontinuous Dedekind complete vector lattice functional calculus as follows.
	
	\begin{definition}\label{D: funcal}
		Let $n\in\N$, and assume that $h\colon\R^n\to\R$ is an upper semicontinuous positively homogeneous function that that is bounded on the unit sphere. Furthermore, let $\Phi$ be a family of sublinear maps such that for each $\bar{x}\in\R^n$, we have $h(\bar{x})=\underset{\phi\in\Phi}{\inf}\ \phi(\bar{x})$.
		Next consider an Archimedean vector lattice $E$, and put $f_1,\dots,f_n\in E$. If
		\begin{itemize}
			\item[(i)] for every $\phi\in\Phi$, $\phi(f_1,\dots,f_n)$ is defined in $E$, and
			\item[(ii)] $\underset{\phi\in\Phi}{\inf}\ \phi(f_1,\dots,f_n)$ exists in $E$,
		\end{itemize}	
		then we say $h(f_1,\dots,f_n)$ \textit{is defined in} $E$ and write
		\[
		h(f_1,\dots,f_n):=\underset{\phi\in\Phi}{\inf}\ \phi(f_1,\dots,f_n).
		\]
		Likewise, consider a lower semicontinuous positively homogeneous function $h\colon\R^n\to\R$ that is bounded on the unit sphere. Let $\Psi$ be a family of real-valued superlinear maps defined on $\R^n$ such that for all $\bar{x}\in\R^n$, $h(\bar{x})=\underset{\psi\in\Psi}{\sup}\ \psi(\bar{x})$.
		If
		\begin{itemize}
			\item[(i)] for every $\psi\in\Psi$, $\psi(f_1,\dots,f_n)$ is defined in $E$, and
			\item[(ii)] $\underset{\psi\in\Psi}{\sup}\ \phi(f_1,\dots,f_n)$
			exists in $E$,
		\end{itemize}
		then we say $h(f_1,\dots,f_n)$ \textit{is defined in} $E$ and write
		\[
		h(f_1,\dots,f_n):=\underset{\psi\in\Psi}{\sup}\ \psi(f_1,\dots,f_n).
		\]
		We refer to this functional calculus as the \textit{semicontinuous Archimedean vector lattice functional calculus}.
	\end{definition}

In order to address the various issues of well-definedness in the semicontinuous Archimedean
vector lattice functional calculus, it proves convenient to utilize the results which address these issues in the Dedekind complete vector lattice case in Section~\ref{S: h fun cal} in the Dedekind completion of an Archimedean vector lattice. For this reason, we first address the issue of these functional calculus calculations in different vector lattices, where one vector lattice is a vector sublattice of the other.

	\begin{proposition}\label{P: subspace well-defined} Let $n\in\N$, let $E$ be an Archimedean vector lattice, and assume $E_0$ is a vector sublattice of $E$. Put $f_1,\dots,f_n\in E_0$. Suppose $h\colon\R^n\to\R$ is a positively homogeneous function which is bounded on the unit sphere.
		\begin{itemize}
			\item[(i)] Assume $h$ is upper semicontinuous with inf-sublinear representation 
			\[
			h(\bar{x})=\underset{\phi\in\Phi}{\inf}\ \phi(\bar{x})\qquad (\bar{x}\in\R^n).
			\]
			If for all $\phi\in\Phi$, $\phi(f_1,\dots,f_n)$ is defined in $E$, and
			\[
			E_0-\underset{\phi\in\Phi}{\inf}\ \phi(f_1,\dots,f_n)
			\]
			exists in $E_0$, then
			\[
			E-\underset{\phi\in\Phi}{\inf}\ \phi(f_1,\dots,f_n)
			\]
			exists in $E$, and
			\[
			E_0-\underset{\phi\in\Phi}{\inf}\ \phi(f_1,\dots,f_n)=E-\underset{\phi\in\Phi}{\inf}\ \phi(f_1,\dots,f_n).
			\]
			\item[(ii)] Suppose $h$ is lower semicontinuous with sup-superlinear representation 
			\[
			h(\bar{x})=\underset{\psi\in\Psi}{\sup}\ \psi(\bar{x})\qquad (\bar{x}\in\R^n).
			\]
			If for each $\psi\in\Psi$, $\psi(f_1,\dots,f_n)$ is defined in $E$, and
			\[
			E_0-\underset{\psi\in\Psi}{\sup}\ \psi(f_1,\dots,f_n)
			\]
			exists in $E_0$, then
			\[
			E-\underset{\psi\in\Psi}{\sup}\ \psi(f_1,\dots,f_n)
			\]
			exists in $E$, and
			\[
			E_0-\underset{\psi\in\Psi}{\sup}\ \psi(f_1,\dots,f_n)=E-\underset{\psi\in\Psi}{\sup}\ \psi(f_1,\dots,f_n).
			\]
		\end{itemize}
		In summary, the value of $h(f_1,\dots,f_n)$ remains the same whether its associated supremum (infimum) in its inf-sublinear (sup-superlinear) representation is calculated in $E_0$ or $E$. 
	\end{proposition}
	
	\begin{proof}
		We only prove (i), noting that (ii) is handled similarly. To this end, suppose that for all $\phi\in\Phi$, $\phi(f_1,\dots,f_n)$ is defined in $E$ and also that $E_0-\underset{\phi\in\Phi}{\inf}\ \phi(f_1,\dots,f_n)$ exists in $E_0$. By Theorem~\ref{T: dc implies hc}, we know that $E^\delta-\underset{\phi\in\Phi}{\inf}\ \phi(f_1,\dots,f_n)$ exists in $E^\delta$. Arguing as in the proof of Proposition~\ref{P: DC subspace well-defined}, we get
		\[
		E_0-\underset{\phi\in\Phi}{\inf}\ \phi(f_1,\dots,f_n)= E^\delta-\underset{\phi\in\Phi}{\inf}\ \phi(f_1,\dots,f_n).
		\]
		This equality shows that $E^\delta-\underset{\phi\in\Phi}{\inf}\ \phi(f_1,\dots,f_n)\in E_0\subseteq E$, and so $E-\underset{\phi\in\Phi}{\inf}\ \phi(f_1,\dots,f_n)$ exists in $E$ and
		\[
		E-\underset{\phi\in\Phi}{\inf}\ \phi(f_1,\dots,f_n)=	E^\delta-\underset{\phi\in\Phi}{\inf}\ \phi(f_1,\dots,f_n)=E_0-\underset{\phi\in\Phi}{\inf}\ \phi(f_1,\dots,f_n).
		\]
	\end{proof}	
	
		Our next proposition states that, given two inf-sublinear (sup-superlinear) representations of an upper (lower) semicontinuous positively homogeneous function which is bounded on the unit sphere, the semicontinuous Archimedean vector lattice functional calculus of Definition~\ref{D: funcal} yields the same output under both representations.
	
	\begin{proposition}\label{P: welldefined avl}
		Let $n\in\N$, and let $h\colon\R^n\to\R$ be a positively homogeneous function that is bounded on the unit sphere. Also let $E$ be an Archimedean vector lattice, and put $f_1,\dots,f_n\in E$.
		\begin{itemize}
			\item[(i)] Suppose $h$ is upper semicontinuous. Let $\Phi$ and $\Xi$ be families of real-valued sublinear maps defined on $\R^n$ such that for all $\bar{x}\in\R^n$,
			\[
			h(\bar{x})=\underset{\phi\in\Phi}{\inf}\ \phi(\bar{x})=\underset{\xi\in\Xi}{\inf}\ \xi(\bar{x}).
			\]
			If for every $\phi\in\Phi$ and all $\xi\in\Xi$, $\phi(f_1,\dots,f_n)$ and $\xi(f_1,\dots,f_n)$ are defined in $E$, and $\underset{\phi\in\Phi}{\inf}\ \phi(f_1,\dots,f_n)$ exists in $E$, then $\underset{\xi\in\Xi}{\inf}\ \xi(f_1,\dots,f_n)$ exists in $E$, and
			\[
			\underset{\phi\in\Phi}{\inf}\ \phi(f_1,\dots,f_n)=\underset{\xi\in\Xi}{\inf}\ \xi(f_1,\dots,f_n).
			\]
			\item[(ii)] Suppose $h$ is lower semicontinuous. Let $\Psi$ and $\Xi$ be families of real-valued superlinear maps defined on $\R^n$ such that for all $\bar{x}\in\R^n$,
			\[
			h(\bar{x})=\underset{\psi\in\Psi}{\sup}\ \psi(\bar{x})=\underset{\xi\in\Xi}{\sup}\ \xi(\bar{x}).
			\]
			If for every $\psi\in\Psi$ and each $\xi\in\Xi$, $\psi(f_1,\dots,f_n)$ and $\xi(f_1,\dots,f_n)$ are defined in $E$, and $\underset{\psi\in\Psi}{\sup}\ \psi(f_1,\dots,f_n)$ exists in $E$, then $\underset{\xi\in\Xi}{\sup}\ \xi(f_1,\dots,f_n)$ exists in $E$, and
			\[
			\underset{\psi\in\Psi}{\sup}\ \psi(f_1,\dots,f_n)=\underset{\xi\in\Xi}{\sup}\ \xi(f_1,\dots,f_n).
			\]
		\end{itemize}
	\end{proposition}
	
	\begin{proof}
		We only prove part (i) of the proposition, since part (ii) is similar. To this end, let $f_1,\dots,f_n\in E$, and suppose that for all $\phi\in\Phi$, and every $\xi\in\Xi$, $\phi(f_1,\dots,f_n), \xi(f_1,\dots,f_n)$ exist in $E$ and also that $\underset{\phi\in\Phi}{\inf}\ \phi(f_1,\dots,f_n)$ exists in $E$. By Theorem~\ref{T: dc implies hc}, we know that both $\underset{\phi\in\Phi}{\inf}\ \phi(f_1,\dots,f_n)$ and $\underset{\xi\in\Xi}{\inf}\ \xi(f_1,\dots,f_n)$ exist in $E^\delta$. Viewing $f_1,\dots,f_n$ as elements in $E^\delta$, then by Proposition~\ref{P: welldefined}, we have
		\[
		E^\delta-\underset{\phi\in\Phi}{\inf}\ \phi(f_1,\dots,f_n)=E^\delta-\underset{\xi\in\Xi}{\inf}\ \xi(f_1,\dots,f_n).
		\]
		However, $E-\underset{\phi\in\Phi}{\inf}\ \phi(f_1,\dots,f_n)$ exists in $E$ by assumption, and by Proposition~\ref{P: subspace well-defined}, we have
		\[
		E-\underset{\phi\in\Phi}{\inf}\ \phi(f_1,\dots,f_n)=E^\delta-\underset{\phi\in\Phi}{\inf}\ \phi(f_1,\dots,f_n).
		\]
		It follows that $E-\underset{\xi\in\Xi}{\inf}\ \xi(f_1,\dots,f_n)$ exists in $E$ and that
		\[
		E-\underset{\phi\in\Phi}{\inf}\ \phi(f_1,\dots,f_n)=E-\underset{\xi\in\Xi}{\inf}\ \xi(f_1,\dots,f_n).
		\]
	\end{proof}
	
	Using a similar proof to the proof of Proposition~\ref{P: welldefined avl}, one can exploit the Dedekind completion and Propositions~\ref{P: cts well defined}\&\ref{P: subspace well-defined} to prove the following result.
	
	\begin{proposition}
		Fix $n\in\N$. Let $h\colon\R^n\to\R$ be a continuous positively homogeneous function, and suppose that for every $\bar{x}\in\R^n$, both
		\[
			h(\bar{x})=\underset{\phi\in\Phi}{\inf}\ \phi(\bar{x})
		\]
		and
		\[
			h(\bar{x})=\underset{\psi\in\Psi}{\sup}\ \psi(\bar{x})
		\]
		hold. Also let $E$ be an Archimedean vector lattice, and put $f_1,\dots,f_n\in E$. If for every $\phi\in\Phi$, $\phi(f_1,\dots,f_n)$ is defined in $E$, for all $\psi\in\Psi$, $\psi(f_1,\dots,f_n)$ is defined in $E$, and either $\underset{\phi\in\Phi}{\inf}\ \phi(f_1,\dots,f_n)$ exists in $E$ or $\underset{\psi\in\Psi}{\sup}\ \psi(f_1,\dots,f_n)$ exists in $E$, then both $\underset{\phi\in\Phi}{\inf}\ \phi(f_1,\dots,f_n)$ and $\underset{\psi\in\Psi}{\sup}\ \psi(f_1,\dots,f_n)$ exist in $E$, and
		\[
		\underset{\phi\in\Phi}{\inf}\ \phi(f_1,\dots,f_n) = \underset{\psi\in\Psi}{\sup}\ \psi(f_1,\dots,f_n).
		\]
	\end{proposition}
	As in the case of the sublinear/superlinear Archimedean vector lattice functional calculus, we are especially attentive to Archimedean vector lattices where the semicontinuous Archimedean vector lattice functional calculus, with respect to a specific semicontinuous real-valued positively homogeneous function defined on $\R^n$, can always be performed. Therefore, we extend (for a sublinear or superlinear map $\xi\colon\R^n\to\R$) the notion of an $\xi$-closed Archimedean vector lattice to include all semicontinuous positively homogeneous functions $h\colon\R^n\to\R$. 
	
	\begin{definition}\label{D: h closed and h-closure}
		Fix $n\in\N$, and let $h\colon\R^n\to\R$ be a semicontinuous positively homogeneous function which is bounded on the unit sphere.
		\begin{itemize}
			\item[(i)]  We call an Archimedean vector lattice $E$ $h$-\textit{closed} if for every $f_1,\dots,f_n\in E$, $h(f_1,\dots,f_n)$ is defined in $E$.
			\item[(ii)] Given an Archimedean vector lattice $E$, we set $E_1:=E$. For $m\in\N$, we define $E_{m+1}$ to be the vector sublattice of $E^\delta$ generated by the set
			\[
			\bigl\{E_m\cup\{h(f_1,\dots,f_n)\ :\ f_1,\dots,f_n\in E_n\}\bigr\}.
			\]
			We define $E^h:=\bigcup_{m\in\N}E_m$ and call $E^h$ the $h$-\textit{closure} of $E$.
		\end{itemize}
	\end{definition}
	
	The concluding proposition of this section, as in the special case dealing with sublinear and superlinear maps, is immediate.
	
	\begin{proposition}
		If $E$ is an Archimedean vector lattice, $n\in\N$, and $h\colon\R^n\to\R$ is a semicontinuous positively homogeneous function which is bounded on the unit sphere, then $E^h$ is an $h$-closed Archimedean vector lattice.
	\end{proposition}

	\section{Comparison to the Continuous Archimedean Vector Lattice Functional Calculus}\label{S: comparisonto'91}
	
	We prove in this section that the semicontinuous Archimedean vector lattice functional calculus introduced in Section~\ref{S: non DC h fun cal} of this paper agrees with the continuous Archimedean vector lattice functional calculus for continuous positively homogeneous functions. Therefore, our functional calculus can be considered as an extension of the continuous version. This claim follows from the following result, which is of interest in its own right.
	
	\begin{theorem}\label{T: swapTandh}
		Assume that $E$ and $F$ are Archimedean vector lattices and that $T\colon E\to F$ is a vector lattice homomorphism. Let $n\in\N$, and suppose that $h\colon\R^n\to\R$ is a positively homogeneous function that is bounded on the unit sphere. Fix $f_1,\dots,f_n\in E$.
		\begin{itemize}
			\item[(i)] If $h$ is upper semicontinuous, $h(f_1,\dots,f_n)$ is defined in $E$, and $h\bigl(T(f_1),\dots,T(f_n)\bigr)$ is defined in $F$, then
			\[
			T\bigl(h(f_1,\dots,f_n)\bigr)\leq h\bigl(T(f_1),\dots,T(f_n)\bigr).
			\]
			\item[(ii)] If $h$ is lower semicontinuous, $h(f_1,\dots,f_n)$ is defined in $E$, and $h\bigl(T(f_1),\dots,T(f_n)\bigr)$ is defined in $F$, then
			\[
			T\bigl(h(f_1,\dots,f_n)\bigr)\geq h\bigl(T(f_1),\dots,T(f_n)\bigr).
			\]
			\item[(iii)] If $h$ is continuous, $h(f_1,\dots,f_n)$ is defined in $E$, and $h\bigl(T(f_1),\dots,T(f_n)\bigr)$ is defined in $F$, then
			\[
			T\bigl(h(f_1,\dots,f_n)\bigr)=h\bigl(T(f_1),\dots,T(f_n)\bigr).
			\]
		\end{itemize}
	\end{theorem}
	
	\begin{proof}
		(i) Assume that $h$ is upper semicontinuous, that $h(f_1,\dots,f_n)$ is defined in $E$, and that $h\bigl(T(f_1),\dots,T(f_n)\bigr)$ is defined in $F$. Suppose $h(\bar{x})=\underset{\phi\in\Phi}{\inf}\ \phi(\bar{x})\quad (\bar{x}\in\R^n)$ provides an inf-sublinear representation for $h$, and fix $\phi_0\in\Phi$. Using the fact that vector lattice homomorphisms are positive maps in the inequality below, and utilizing Proposition~\ref{T: phi hom interchange avl} in the last identity below, we obtain
		\begin{align*}
			T\bigl(h(f_1,\dots,f_n)\bigr)&=T\left(\underset{\phi\in\Phi}{\inf}\phi(f_1,\dots,f_n)\right)\\
			&\leq T\bigl(\phi_0(f_1,\dots,f_n)\bigr)\\
			&=\phi_0\bigl(T(f_1),\dots,T(f_n)\bigr).
		\end{align*}
		It follows that
		\[
		T\bigl(h(f_1,\dots,f_n)\bigr)\leq h\bigl(T(f_1),\dots,T(f_n)\bigr).
		\]
	
	(ii) Analogous to the proof of (i).
	
	(iii) Follows from (i) and (ii) above, noting that a continuous function is both upper semicontinuous and lower semicontinuous.
\end{proof}
	
	As an almost immediate corollary, we have that our functional calculus defined in Section~\ref{S: non DC h fun cal} agrees with the continuous Archimedean vector lattice functional calculus in \cite{BusdPvR}. We will utilize the following notation in the formal statement and proof of this claim.
	
	\begin{notation}
		Put $n\in\N$, and let $E$ be an Archimedean vector lattice. Assume $h\colon\R^n\to\R$ is a continuous positively homogeneous function, and let $f_1,\dots,f_n\in E$. In order to avoid ambiguity in the statement and proof of Corollary~\ref{C: swapTandh}, we will write $(c)\sim h(f_1,\dots,f_n)$ to denote $h(f_1,\dots,f_n)$ defined by the continuous Archimedean vector lattice functional calculus of \cite{BusdPvR} and signify $h(f_1,\dots,f_n)$ defined by the semicontinuous Archimedean vector lattice functional calculus of Definition~\ref{D: funcal} by $(s)\sim h(f_1,\dots,f_n)$.
	\end{notation}
	
	\begin{corollary}\label{C: swapTandh}
		Fix $n\in\N$. Let $E$ be an Archimedean vector lattice, suppose $h\colon\R^n\to\R$ is a continuous positively homogeneous function, and put $f_1,\dots,f_n\in E$. It is true that $(c)\sim h(f_1,\dots,f_n)$ is defined in $E$ if and only if $(s)\sim h(f_1,\dots,f_n)$ is defined in $E$, and in this case we have
		\[
		(c)\sim h(f_1,\dots,f_n)=(s)\sim h(f_1,\dots,f_n).
		\]
	\end{corollary}
	
	\begin{proof}
	By \cite[Lemma~39.2]{LuxZan1}, $E^\delta$ is uniformly complete, and hence $(c)\sim h(f_1,\dots,f_n)$ is defined in $E^\delta$ (see \cite[Theorem~3.7]{BusvR}). It follows from Theorem~\ref{T: dc implies hc} that $(s)\sim h(f_1,\dots,f_n)$ is defined in $E^\delta$ as well.
		
		Denote the vector sublattice of $E^\delta$ generated by
		\[
		\{f_1,\dots,f_n, (c)\sim h(f_1,\dots,f_n), (s)\sim h(f_1,\dots,f_n)\}
		\]
		by $E_0$. Let $\omega\colon E_0\to\R$ be a vector lattice homomorphism. It follows from \cite[Theorem~3.11]{BusSch} that
		\[
		\omega\Bigl((c)\sim h(f_1,\dots,f_n)\Bigr)=h\Bigl(\omega(f_1),\dots,\omega(f_n)\Bigr).
		\]
		By Theorem~\ref{T: swapTandh}, we have
		\[
		h\Bigl(\omega(f_1),\dots,\omega(f_n)\Bigr)=\omega\Bigl((s)\sim h(f_1,\dots,f_n)\Bigr).
		\]
		Hence we obtain
		\[
		\omega\Bigl((c)\sim h(f_1,\dots,f_n)\Bigr)=\omega\Bigl((s)\sim h(f_1,\dots,f_n)\Bigr).
		\]
		Since the set of all nonzero vector lattice homomorphisms $\omega\colon E_0\to\R$ separate the points of $E_0$ \cite[remark 1.2(ii), Theorem~2.2]{BusvR}, we conclude that in $E^\delta$ we have
		\[
		(c)\sim h(f_1,\dots,f_n)=(s)\sim h(f_1,\dots,f_n).
		\]
Next observe that if $(c)\sim h(f_1,\dots,f_n)$ is defined in $E$, then by \cite[Theorem~3.11]{BusSch} (and taking $T\colon E\to E^\delta$ to be the inclusion map), the value of $(c)\sim h(f_1,\dots,f_n)$ remains unchanged if the functional calculus is performed in $E$ or $E^\delta$. Likewise, if $(s)\sim h(f_1,\dots,f_n)$ is defined in $E$, then by Proposition~\ref{P: subspace well-defined}, the value of $(s)\sim h(f_1,\dots,f_n)$ calculated in $E$ equals the value of $(s)\sim h(f_1,\dots,f_n)$ calculated in $E^\delta$. It follows that $(c)\sim h(f_1,\dots,f_n)$ is defined in $E$ if and only if $(s)\sim h(f_1,\dots,f_n)$ is defined in $E$, in which case
		\[
		(c)\sim h(f_1,\dots,f_n)=(s)\sim h(f_1,\dots,f_n)
		\]
		holds in $E$.
		
	\end{proof}
	
	Hence, as previously mentioned, the functional calculus of Section~\ref{S: non DC h fun cal} in this paper can be viewed as an extension of the functional calculus of \cite{BusdPvR}.
	
	\section{Saddle Representations}\label{S: saddlerepresnentations}
	
	We provide alternative expressions for \textit{continuous} positively homogeneous functions defined via our semicontinuous Dedekind complete vector lattice functional calculus, based on the saddle representations of such functions. The advantage to these other possible formulations is that they are less abstract, in that the sublinear and superlinear maps used in our functional calculus can be replaced simply by real number scalars. We begin with the definition of a saddle representation of a continuous positively homogeneous function.
	
	\begin{definition}\cite[Section~2]{GorTraf}
		Let $n\in\N$, and let $h\colon\R^n\to\R$ be a continuous, positively homogeneous function. We say that $h$ has a \textit{saddle representation} by linear functions if there exists a two-index family $\{\overline{a^{\phi\psi}}\in\R^n\ :\ \phi\in\Phi, \psi\in\Psi\}$ such that
		\[
		h(\bar{x})=\underset{\phi\in\Phi}{\inf}\ \underset{\psi\in\Psi}{\sup}\sum_{i=1}^{n}a_i^{\phi\psi}x_i=\underset{\psi\in\Psi}{\sup}\ \underset{\phi\in\Phi}{\inf}\sum_{i=1}^{n}a_i^{\phi\psi}x_i\qquad (\bar{x}\in\R^n).
		\]
	\end{definition}
	
	The next theorem confirms that every continuous positively homogeneous real-valued function defined on $\R^n$ has a saddle representation.
	
	\begin{theorem}\cite[Theorem~1]{GorTraf}\label{T:GorTrafsaddlerep}
		For $n\in\N$ and a function $h\colon\R^n\to\R$, it is true that $h$ is continuous and positively homogeneous if and only if $h$ has a saddle representation.
	\end{theorem}
	
	Next, we provide a saddle representation for functions on a Dedekind complete vector lattice $E$ defined by the functional calculus of Section~\ref{S: h fun cal}. The proof of the following result is an adaptation of the proof of \cite[Theorem~1]{GorTraf}.
	
	\begin{theorem}\label{T: saddlerepresentation}
		Fix $n\in\N$, and assume $E$ is a Dedekind complete vector lattice. Let $h\colon \R^n\to\R$ be a positively homogeneous and continuous function. Suppose $f_1,\dots,f_n\in E$. If 
		\[
		\{\overline{a^{\phi\psi}}\in\R^n\ :\ \phi\in\Phi, \psi\in\Psi\}
		\]
		provides a saddle representation for $h$, that is, if
		\[
		h(\bar{x})=\underset{\phi\in\Phi}{\inf}\ \underset{\psi\in\Psi}{\sup}\sum_{i=1}^{n}a_i^{\phi\psi}x_i=\underset{\psi\in\Psi}{\sup}\ \underset{\phi\in\Phi}{\inf}\sum_{i=1}^{n}a_i^{\phi\psi}x_i\qquad (\bar{x}\in\R^n),
		\]
		then
		\[
		h(f_1,\dots,f_n)=\underset{\phi\in\Phi}{\inf}\ \underset{\psi\in\Psi}{\sup}\sum_{i=1}^{n}a_i^{\phi\psi}f_i=\underset{\psi\in\Psi}{\sup}\ \underset{\phi\in\Phi}{\inf}\sum_{i=1}^{n}a_i^{\phi\psi}f_i.
		\]
	\end{theorem}
	
	\begin{proof}
		Since $h$ is continuous, we know $h$ is both upper and lower semicontinuous as well as bounded on the unit sphere. Let $\Phi$, respectively $\Psi$, be a family of sublinear, respectively superlinear maps, which satisfies \eqref{eq: inf-sublinear rep}, respectively \eqref{eq: sup-superlinear rep}. From the proof of \cite[Theorem~1]{GorTraf}, we have for any $\phi\in\Phi$ and $\psi\in\Psi$ that there exists a linear map $(\bar{x})\mapsto\sum_{i=1}^{n}a_i^{\phi\psi}x_i$ for which
		\[
		\psi(\bar{x})\leq\sum_{i=1}^{n}a_i^{\phi\psi}x_i\leq \phi(\bar{x})\qquad (\bar{x}\in\R^n).
		\]
		Fix $\phi_0\in\Phi$ and $\psi_0\in\Psi$, and let $E_0$ be the vector sublattice of $E$ generated by
		\[
		\{f_1,\dots,f_n, \phi_0(f_1,\dots,f_n), \psi_0(f_1,\dots,f_n)\},
		\]
		and let $\omega\colon E_0\to\R$ be a vector lattice homomorphism. Utilizing Proposition~\ref{T: phi hom interchange} in each of the equalities below, we have
		\begin{align*}
			\omega\bigl(\psi_0(f_1,\dots,f_n)\bigr)&=\psi_0\bigl(\omega(f_1),\dots,\omega(f_n)\bigr)\\
			&\leq\sum_{i=1}^{n}a_i^{\phi_0\psi_0}\omega(f_i)\\
			&\leq \phi_0\bigl(\omega(f_1),\dots,\omega(f_n)\bigr)\\
			&=\omega\bigl(\phi_0(f_1,\dots,f_n)\bigr).
		\end{align*}
		Noting that $\sum_{i=1}^{n}a_i^{\phi_0\psi_0}\omega(f_i)=\omega\left(\sum_{i=1}^{n}a_i^{\phi_0\psi_0}f_i\right)$, we obtain
		\[
		\omega\bigl(\psi_0(f_1,\dots,f_n)\bigr)\leq\omega\left(\sum_{i=1}^{n}a_i^{\phi_0\psi_0}f_i\right)\leq\omega\bigl(\phi_0(f_1,\dots,f_n)\bigr).
		\]
		Since the set of all nonzero vector lattice homomorphisms $\omega\colon E_0\to\R$ separate the points of $E_0$ \cite[remark 1.2(ii), Theorem~2.2]{BusvR}, we get that
		\begin{equation}\label{eq: compound ineq}
		\psi_0(f_1,\dots,f_n)\leq\sum_{i=1}^{n}a_i^{\phi_0\psi_0}f_i\leq \phi_0(f_1,\dots,f_n).
		\end{equation}
		Taking the infimum over all $\phi\in\Phi$ in \eqref{eq: compound ineq}, which is possible since $E$ is Dedekind complete, we obtain
		\[
		\psi_0(f_1,\dots,f_n)\leq \underset{\phi\in\Phi}{\inf}\ \sum_{i=1}^{n}a_i^{\phi\psi_0}f_i\leq \underset{\phi\in\Phi}{\inf}\ \phi(f_1,\dots,f_n)=h(f_1,\dots,f_n).
		\]
		Next, taking the supremum over all $\psi\in\Psi$ (which is again possible due to the assumption that $E$ is Dedekind complete), we get
		\begin{align*}
			h(f_1,\dots,f_n)&=\underset{\psi_\in\Psi}{\sup}\ \psi(f_1,\dots,f_n)\\
			&\leq\underset{\psi\in\Psi}{\sup}\ \underset{\phi\in\Phi}{\inf}\ \sum_{i=1}^{n}a_i^{\phi\psi}f_i\\
			&\leq \underset{\phi\in\Phi}{\inf}\ \phi(f_1,\dots,f_n)\\
			&=h(f_1,\dots,f_n).
		\end{align*}
		Therefore, we have
		\[
		h(f_1,\dots,f_n)=\underset{\psi\in\Psi}{\sup}\ \underset{\phi\in\Phi}{\inf}\ \sum_{i=1}^{n}a_i^{\phi\psi}f_i.
		\]
		Repeating this argument, except taking the supremum over all $\psi\in\Psi$ first and then taking the infimum over all $\phi\in\Phi$ second, we get
		\[
		h(f_1,\dots,f_n)=\underset{\phi\in\Phi}{\inf}\ \underset{\psi\in\Psi}{\sup}\ \sum_{i=1}^{n}a_i^{\phi\psi}f_i
		\]
		as well.
	\end{proof}
	
	In conclusion of this section, Theorem~\ref{T: saddlerepresentation} shows that saddle representations provide perhaps simpler formulas than Definition~\ref{D: DC h funcal} for continuous positively homogeneous defined via our functional calculus in the setting of Dedekind complete vector lattices.
	
	\section{Examples}\label{S: examples}
	
	We conclude this paper by providing some examples of how the semicontinuous Archimedean vector lattice functional calculus can be exploited to define some semicontinuous functions on Archimedean vector lattices.
	
	\begin{example}\label{E: ex1}
		Define $h\colon\R^2\to\R$ by
		\[
		h(x,y):=\begin{cases}
			x+y & \text{if}\ x,y\geq 0\\
			0 & \text{else}
		\end{cases}.
		\]
		Observe that $h$ is upper semicontinuous, positively homogeneous, and bounded on the unit sphere. However, for every $x>0$, $h$ is discontinuous at $(x,0)$. Moreover, $h$ is also discontinuous at $(0,y)$ for all $y>0$. Hence $h$ is outside the scope of the continuous Archimedean vector lattice functional calculus of \cite{BusdPvR}.
		
		However, by providing an inf-sublinear representation for $h$, we can utilize our functional calculus to define $h$ in any $h$-closed Archimedean vector lattice. To this end,
		for every $m,n\in\N$ and all $(x,y)\in\R^2$, define the sublinear map
		\[
		\phi_{m,n}(x,y):=\max\{mx+ny,0\}.
		\]
		Note that if $x,y\geq 0$, then
		\[
		\underset{m,n\in\N}{\inf}\ \max\{mx+ny,0\}=x+y=h(x,y).
		\]
		If $x,y<0$ however, then
		\[
		\underset{m,n\in\N}{\inf}\ \max\{mx+ny,0\}=0=h(x,y).
		\]
		Next assume that one of $x$ and $y$ is negative and the other is zero. Without loss of generality, say $x<0=y$. Then
		\[
		\underset{m,n\in\N}{\inf}\ \max\{mx+ny,0\}=\underset{m,n\in\N}{\inf}\ \max\{mx+n\cdot0,0\}=0=h(x,y).
		\]
		Finally, suppose that one of $x$ and $y$ is negative and the other is positive. Without loss of generality, we assume $x<0<y$. Then there exists $m_0\in\N$ such that $m_0(-x)>y$, that is, $m_0x+y<0$. We thus obtain
		\[
		\underset{m,n\in\N}{\inf}\ \max\{mx+ny,0\}\leq\max\{m_0x+y,0\}=0=h(x,y).
		\]
		Therefore,
		\[
		h(x,y)=\underset{m,n\in\N}{\inf}\ \max\{mx+ny,0\}=\underset{m,n\in\N}{\inf}\ \bigl((mx+ny)^+\bigr)\qquad \bigl((x,y)\in\R^2\bigr)
		\]
		is an inf-sublinear representation of $h$. Hence we can define $h$ on the two-fold Cartesian product of any $h$-complete Archimedean vector lattice $E$ by
		\[
		h(f,g):=\underset{m,n\in\N}{\inf}\ \bigl((mf+ng)^+\bigr)\qquad (f,g\in E).
		\]
	\end{example}
	
	\begin{example}\label{E: ex2}
		Define $h\colon\R^2\to\R$ by
		\[
		h(x,y):=\begin{cases}
			x & \text{if}\ x,y>0\\
			y & \text{if}\ y<0\\
			0 & \text{else}
		\end{cases}.
		\]
	\end{example}
	Note that $h$ is positively homogeneous, lower semicontinuous (but not continuous), and bounded on the unit sphere.
	
	Next define for every $\lambda\in\{0,1\}$ and all $n\in\N$ the superlinear map
	\[
	\psi_{\lambda, n}(x,y)=\min\{\lambda x, ny\}.
	\]
	Suppose $x,y>0$. Then we get
	\[
	\underset{\lambda\in\{0,1\}, n\in\N}{\sup}\ \min\{\lambda x, ny\}=\underset{n\in\N}{\sup}\ \min\{x,ny\}=x=h(x,y).
	\]
	On the other hand, if $x,y<0$, we have
	\[
	\underset{\lambda\in\{0,1\}, n\in\N}{\sup}\ \min\{\lambda x, ny\}=\underset{n\in\N}{\sup}\ \min\{0, ny\}=y=h(x,y),
	\]
	and if $y<0\leq x$, we get
	\[
	\underset{\lambda\in\{0,1\}, n\in\N}{\sup}\ \min\{\lambda x, ny\}=\underset{n\in\N}{\sup}\ ny =y=h(x,y).
	\]
	In the case that $x\leq 0<y$, we have
	\[
	\underset{\lambda\in\{0,1\}, n\in\N}{\sup}\ \min\{\lambda x, ny\}=\underset{\lambda\in\{0,1\}}{\sup}\ \lambda x=0=h(x,y),
	\]
	and finally, if $y=0$, we obtain
	\[
	\underset{\lambda\in\{0,1\}, n\in\N}{\sup}\ \min\{\lambda x, ny\}=\underset{\lambda\in\{0,1\}}{\sup}\ \min\{\lambda x, 0\}=0=h(x,y).
	\]
	Hence
	\[
	h(x,y)=\underset{\lambda\in\{0,1\}, n\in\N}{\sup}\ \min\{\lambda x, ny\}=\underset{\lambda\in\{0,1\}, n\in\N}{\sup}\ \bigl((\lambda x)\wedge(ny)\bigr)\qquad \bigl((x,y)\in\R^2\bigr)
	\]
	is a sup-superlinear representation of $h$. We conclude that $h$ can be defined on any Cartesian square of an $h$-complete Archimedean vector lattice $E$ by
	\[
	h(f,g):=\underset{\lambda\in\{0,1\}, n\in\N}{\sup}\ \bigl((\lambda f)\wedge(ng)\bigr)\qquad (f,g\in E).
	\]

\end{document}